\pdfoutput=1 
\documentclass{article}
\usepackage[dvipsnames]{xcolor} %for colored text
\usepackage{amsmath} %for \text
\usepackage{amsthm} %for \proof
\usepackage{dsfont} %for \mathds
\usepackage{fullpage}

\usepackage{mathtools}%for mathclap etc.

\usepackage{bm} %bold letters in math mode
\usepackage{graphicx}
\usepackage{subfigure}

\theoremstyle{plain}% gives italic text
\newtheorem{theorem}{Theorem}[section]
\newtheorem{lemma}[theorem]{Lemma}
\newtheorem{corollary}[theorem]{Corollary}
\newtheorem{proposition}[theorem]{Proposition}
\newtheorem{test}{Test}

\theoremstyle{definition}% gives upshape text
\newtheorem{remark}[theorem]{Remark}

\newtheorem{definition}[theorem]{Definition}

\numberwithin{equation}{section}

\newcommand{\comp}{\mathds{C}}
\newcommand{\Ee}{\mathds{E}}
\newcommand{\nat}{\mathds{N}}
\newcommand{\real}{\mathds{R}}
\newcommand{\Pp}{\mathds{P}}
\newcommand{\hN}{\mbox{}^{\scriptscriptstyle N}\kern-1.5pt}
\newcommand{\MN}{\mbox{}^{\scriptscriptstyle N}\kern-2.5pt{M}}
\newcommand{\RN}{\mbox{}^{\scriptscriptstyle N}\kern-1.5pt{\mathcal{R}}}
\newcommand{\MsN}{\mbox{}^{\scriptscriptstyle N}\kern-2pt{\mathcal{M}}}
\newcommand{\MNo}{\mbox{}^{\scriptscriptstyle N}\kern-2.5pt\overline{M}}
\newcommand{\sampleN}{}

\newcommand{\scalp}[2]{#1\cdot#2} %scalar product

\newcommand{\sbullet}{{\scriptscriptstyle\bullet}}

\newcommand{\GC}{\mathcal{G}}

\newcommand{\Askript}{\mathcal{A}}

\newcommand{\Cskript}{\mathcal{C}}
\newcommand{\Fskript}{\mathcal{F}}
\newcommand{\Rskript}{\mathcal{R}}
\newcommand{\Mskript}{\mathcal{M}}

\newcommand{\One}{\mathds{1}}
\newcommand{\norm}[1]{\left\Vert#1\right\Vert}
\newcommand\doverline[1]{\overline{\overline{#1}}}

\newcommand{\dd}{\mathrm{d}}
\newcommand{\ddB}{\mathrm{d}B}
\newcommand{\ddr}{\mathrm{d}r}
\newcommand{\dds}{\mathrm{d}s}
\newcommand{\ddt}{\mathrm{d}t}

\DeclareMathOperator{\Cov}{Cov}

\DeclareMathOperator{\Ce}{C}
\DeclareMathOperator{\trace}{trace}
\DeclareMathOperator{\sgn}{sgn}

\newcommand{\VN}{\mbox{}^{\scriptscriptstyle N}\kern-1.5pt{V}}
\newcommand{\FN}{\mbox{}^{\scriptscriptstyle N}\kern-1.5pt{F}}
\newcommand{\fN}{\mbox{}^{\scriptscriptstyle N}\kern-1.5pt{f}}

\newcommand{\ii}{\mathrm{i}}
\newcommand{\ee}{\mathrm{e}}

\renewcommand\labelenumi{\textup{\alph{enumi})}}
\renewcommand\theenumi\labelenumi

\begin{document}

\title{Distance multivariance:\\ New dependence measures for random vectors\footnote{Accepted for publication in Annals of Statistics}}
%\runtitle{Distance Multivariance}

\author{B.~B\"{o}ttcher, M.~Keller-Ressel, R.\,L.~Schilling\\ \\
TU Dresden \\ Fakult\"at Mathematik\\ Institut f\"{u}r Mathematische Stochastik\\ 01062 Dresden, Germany}

\date{October 2018}
\maketitle

\begin{abstract}
    We introduce two new measures for the dependence of $n \ge 2$ random variables:  \emph{distance multivariance} and \emph{total distance multivariance}. Both measures are based on the weighted $L^2$-distance of quantities related to the characteristic functions of the underlying random variables. These extend distance covariance (introduced by Sz\'ekely, Rizzo and Bakirov) from pairs of random variables to $n$-tuplets of random variables. We show that total distance multivariance can be used to detect the independence of $n$ random variables and has a simple finite-sample representation in terms of distance matrices of the sample points, where distance is measured by a continuous negative definite function. Under some mild moment conditions, this leads to a test for independence of multiple random vectors which is consistent against all alternatives.
\end{abstract}



\tableofcontents
\allowdisplaybreaks

\section{Introduction and related work}
Distance multivariance $M_\rho(X_1, X_2,\dots, X_n)$ and total distance multivariance $\overline M_\rho(X_1, X_2, \dots, X_n)$ are new measures for the dependence of  random variables $X_1, \dots, X_n$. They are closely related to distance covariance, as introduced by Sz\'ekely, Rizzo and Bakirov \cite{SzekRizzBaki2007, SzekRizz2009} and its generalizations presented in \cite{part1}. Distance multivariance inherits many of the features of distance covariance; in particular, see Theorem~\ref{thm:independence} below,
\begin{itemize}
\item
    $M_\rho(X_1, \dots, X_n)$ and $\overline M_\rho(X_1, \dots, X_n)$ are defined for random variables\\ $X_1, \dots$, $X_n$ with values in spaces of arbitrary dimensions $\real^{d_1}, \dots, \real^{d_n}$;
\item
    if each subfamily of $X_1, \dots, X_n$ with $n-1$ elements is independent, $M_\rho(X_1, \dots, X_n) = 0$ characterizes the independence of $X_1, \dots, X_n$;
\item
    $\overline M_\rho(X_1, \dots, X_n) = 0$ characterizes the independence of $X_1, \dots, X_n$.
\end{itemize}
We emphasize that measuring the dependence of $n$ random variables is different from measuring their pairwise dependence, and for this reason bivariate dependence measures, such as distance covariance,  cannot be used directly to detect overall independence. A classical example, Bernstein's coins, is discussed in Section~\ref{ex:bernstein}. The extension of distance covariance to more than two random variables was addressed in a short paragraph in Bakirov and Sz\'ekely \cite{BakiSzek2011}. Our approach is different from the approach suggested in \cite{BakiSzek2011}; it is, in fact, closer to the two approaches that were advised against in \cite{BakiSzek2011}. We will discuss and compare these approaches in greater detail in Section~\ref{sec:BS11}, once the necessary concepts have been introduced. Recently, Yao \emph{et al.}\ \cite{YaoZhanShao2017} introduced measures for pairwise dependence based on distance covariance. In contrast, distance multivariance does not only detect pairwise dependence, but any type of multivariate dependence. Jin and Matteson \cite{JinMatt2017} present measures for multivariate independence which also use distance covariance. The resulting exact estimators are computationally more complex than those of distance multivariance; \cite{Boet2017} shows that the approximate estimators of \cite{JinMatt2017} have less empirical power but are computationally of the same order as distance multivariance.

Another line of research considers dependence measures based on reproducing kernel Hilbert spaces, notably the Hilbert-Schmidt independence criterion (HSIC) of \cite{GretBousSmolScho2005}, which has been shown to be equivalent to distance covariance in \cite{SejdSripGretFuku2013}. Subsequently,  HSIC has been extended from a bivariate dependence measure to a multivariate dependence measure, $\mathrm{dHSIC}$, in \cite{PfisBuehSchoPete2017}. We compare $\mathrm{dHSIC}$ to distance multivariance in Section~\ref{sec:dhsic}.

Similar to distance covariance in \cite{SzekRizzBaki2007} and its generalizations given in \cite{part1}, distance multivariance can be defined as a weighted $L^2$-norm of quantities related to the characteristic functions of $X_1, \dots, X_n$, cf.~Definition~\ref{def:multivariance} below. There are, however, further definitions of distance multivariance which are equivalent up to moment conditions. In particular, multivariance can be equivalently defined as \emph{Gaussian multivariance} by evaluating a Gaussian random field at the instances $(X_1, \dots, X_n)$ and taking certain expectations, see Section~\ref{sub:gaussian}. This generalizes Sz\'ekely-and-Rizzo's \cite[Def.~4]{SzekRizz2009} Brownian covariance which is recovered using $n=2$ and multiparameter Brownian motion as random field.

The sample versions of both distance multivariance and total distance multivariance have simple expressions in terms of the distance matrices of the sample points; this means that we can compute these statistics efficiently even for large samples and in high dimensions. In concrete terms, as we show in Theorem~\ref{thm:sample}, the square of the distance multivariance computed from samples $\bm{x}^{(1)},\dots, \bm{x}^{(N)}$ of the random vector $\bm{X} = (X_1, \dots, X_n)$ can be written as
\begin{equation*}
    \MN^2_\rho(\bm{x}^{(1)},\dots, \bm{x}^{(N)})
    = \frac{1}{N^2} \sum_{j,k=1}^N (A_1)_{jk} \cdot \ldots \cdot (A_n)_{jk}
\end{equation*}
where the $A_i$ are doubly centred distance matrices of the sample points of $X_i$, i.e.~$A_i := - CB_iC$ where $C$ is the centering matrix $C = I - \tfrac{1}{N}\One$, $\One=(1)_{j,k=1,\dots, N}$, $I=(\delta_{jk})_{j,k=1,\dots, N}$, and $B_i$ are the distance matrices of the sample points. The square of the sample \emph{total} distance multivariance has a similar form
\begin{equation*}
    \MNo^2_\rho(\bm{x}^{(1)},\dots, \bm{x}^{(N)})
    = \frac{1}{N^2} \sum_{j,k=1}^N (1+(A_1)_{jk}) \cdot \ldots \cdot (1+ (A_n)_{jk}) - 1.
\end{equation*}
The (quasi-)distance that is used to compute $B_i$ can be chosen, under mild restrictions, from the class of real-valued continuous negative definite functions, cf.~\cite[Ch.~II]{BerFor}, \cite[Sec.~3.2]{Jaco2001}. In particular, we may use Euclidean and $p$-Minkowski distances with exponent $p \in (1,2]$. In the bivariate case, and using Euclidean distance, the sample distance covariance of Sz\'ekely and Rizzo \cite[Def.~3]{SzekRizz2009} is recovered.

Finally, we show in Theorems~\ref{thm:Mdistconv} and \ref{thm:Mdconv} asymptotic properties of sample distance multivariance as $N$ tends to infinity; these results are multivariate analogues of those in \cite[Thm.~5]{SzekRizz2009}. Based on these results, we formulate two new distribution-free tests for the joint independence of $n$ random variables in Section~\ref{sec:test}. These tests are conservative, and a resampling approach can be used to construct tests achieving the nominal size; further results in this direction can be found in \cite{Boet2017}.  The paper concludes in Section~\ref{ex:bernstein} with an extended example based on \emph{Bernstein's coins},  which demonstrates numerically that (total) distance multivariance is able to distinguish between pairwise independence and higher-order dependence of random variables. The example also illustrates the practical validity of the two tests that are proposed. A further example with sinusoidal dependence is discussed, illustrating the influence of the underlying distance on the dependence measure.

For the immediate use of distance multivariance in applications all necessary functions are provided in the R package \texttt{multivariance}, \cite{Boett2017R-1.0.5}.

\section{Preliminaries}\label{sec:prelim}
We consider a $d$-dimensional random vector $\bm{X} = (X_1, \dots, X_n)$, whose components $X_i$ are random variables taking values in $\real^{d_i}$, $i=1,\dots, n$, and where $d = d_1 + \dots + d_n$. The characteristic function of $X_i$ is denoted by
$$
    f_{X_i}(t_i) := \Ee\ee^{\ii \scalp{X_i}{t_i}}, \quad t_i \in \real^{d_i},
$$
and we write $\bm{t} = (t_1, \dots, t_n)$. In order to define the distance multivariance of $(X_1, \dots, X_n)$, we use \emph{L\'evy measures} $\rho_i$, i.e.\ Borel measures $\rho_i$ defined on $\real^{d_i}\setminus\{0\}$ such that
\begin{equation}\label{eq:levy_integrability}
    \int_{\real^{d_i}\setminus\{0\}} \min\{|t_i|^2,1\}\,\rho_i(\ddt_i) <\infty.
\end{equation}
Note that the measures $\rho_i$ need not be finite. Such measures appear in the L\'evy--Khintchine representation of infinitely divisible distributions, see \cite{sato1999levy}.
Throughout this paper we assume that $\rho_i$, $i=1,\dots,n$ are symmetric L\'evy measures with full topological support, cf.~\cite[Def.~2.3]{part1}, and we set $\rho := \rho_{d_1}\otimes\dots\otimes\rho_{d_n}$. To keep notation simple, we write $\int\dots\,\rho_i(\ddt_i)$ and $\int_{\real^{d_i}}\dots\,\rho_i(\ddt_i)$ instead of the formally correct $\int_{\real^{d_i}\setminus\{0\}}\dots\,\rho_i(\ddt_i)$.

\begin{definition}
    Let $(X_i)_{i=1,\dots, n}$ be random variables with values in $\real^{d_i}$ and let the measures $\rho_i$ be given as above. With $\rho := \rho_1 \otimes \dots \otimes \rho_n$, we define

    \smallskip\textup{a)}
    \emph{Distance multivariance} $M_\rho \in [0,\infty]$ by
    \begin{equation}\label{def:multivariance}
        M_\rho^2(X_1,\dots,X_n)
        :=  \int_{\real^d} \left|\Ee\left( \prod_{i=1}^n \left(\ee^{\ii \scalp{X_i}{t_i}}- f_{X_i}(t_i)\right)\right)\right|^2 \rho(\ddt_1, \dots, \ddt_n),
    \end{equation}

    \smallskip\textup{b)}
    \emph{Total distance multivariance} $\overline M_\rho \in [0,\infty]$ by
    \begin{equation}\label{def:total_multivariance}
        \overline M_\rho^2(X_1,\dots,X_n)
        :=  \sum_{\substack{1\leq i_1< \dots < i_m \leq n\\2 \leq m \leq n}} M_{\bigotimes_{j=1}^m\rho_{i_j}}^2(X_{i_1},\dots,X_{i_m}).
    \end{equation}
\end{definition}
\begin{remark}
\ \ a) \
    Using the tensor product for functions
    $$
        \left(g_1 \otimes \dots \otimes g_n\right) (x_1, \dots, x_n) = g_1(x_1) \cdot\ldots\cdot g_n(x_n),
    $$
    distance multivariance can be written in a compact way as
    \begin{equation}\label{eq:M_rho_L2}
        M_\rho(X_1,\dots,X_n)
        = \left\| \Ee\left[ \bigotimes_{i=1}^n \left(\ee^{\ii \scalp{X_i}{\sbullet}}- f_{X_i}(\sbullet)\right)\right]\right\|_{L^2(\rho)}.
\end{equation}
Thus, distance multivariance is the weighted $L^2$-norm of a quantity related to the characteristic functions of the $X_i$, analogous to the definition of distance covariance in Sz\'ekely, Rizzo and Bakirov \cite[Def.~1]{SzekRizzBaki2007}.

\smallskip b) \
    Both $M_\rho$ and $\overline M_\rho$ are always well-defined in $[0,+\infty]$: For each $\bm{t} = (t_1, \dots, t_n)$ the product appearing in the integrand of \eqref{def:multivariance} can be bounded in absolute value by $2^n$; therefore, the expectation exists. The integrand of the $\rho$-integral is positive, and so the integral is always well-defined in $[0,+\infty]$. Just as in the bivariate case, see
    \cite[Thm.~3.7, Rem.~3.8]{part1},
    we need moment conditions on the random variables $X_i$ to guarantee finiteness of $M_\rho$ and $\overline M_\rho$, see Proposition~\ref{prop:representation} below.

\smallskip c) \
    At first sight, total distance multivariance seems to suffer from a computational curse of dimension, since the sum \eqref{def:total_multivariance} extends over all subfamilies (comprising at least two members) of $(X_1, \dots, X_n)$, i.e.\  $2^n - 1 - n$ terms are summed. We will, however, show in Theorem~\ref{thm:sample}, that the finite sample version of $\overline M_\rho$ has the same computational complexity as $M_\rho$ and its computation requires only $\mathcal{O}(nN^2)$ operations given a sample of size $N$.
\end{remark}

Each L\'evy measure $\rho_i$ uniquely defines a real-valued \emph{continuous negative definite function}
\begin{equation}\label{eq:cndf}
    \psi_i(y_i) := \int_{\real^{d_i}} \left(1-\cos(y t_i)\right) \,\rho_i(\ddt_i) \quad \text{for\ \ } y_i\in \real^{d_i},
\end{equation}
see e.g.~\cite[Cor.~3.7.9]{Jaco2001}.
The functions $\psi_i$ will play a key role in the finite-sample representation of distance multivariance and also appear in moment conditions. They are also the reason for the terms \emph{distance} multivariance (and \emph{distance} covariance, cf.~\cite{SzekRizz2009}), since  $\psi_i$  yields well-known distance functions (and in many cases norms) in several important special cases. In particular, $x\mapsto |x|_{}^\alpha$ where $|\cdot|_{}$ is the standard $d_i$-dimensional Euclidean norm and $\alpha \in (0,2)$, can be represented using
$$
    \rho_i(\ddt_i) = c_{\alpha,d_i}\, |t_i|_{}^{-d_i - \alpha}\, \ddt_i,
    \quad \alpha \in  (0,2),
    \quad c_{\alpha,d_i} =  \frac{\alpha 2^{\alpha-1} \Gamma\left(\tfrac{\alpha+d_i}{2}\right)}{\pi^{d_i/2} \Gamma\left(1-\tfrac{\alpha}{2}\right)},
$$
since
$$
    |y_i|_{}^\alpha = c_{d_i, \alpha} \int_{\real^{d_i}} \left(1-\cos \scalp{y}{t_i}\right) \,\frac{\ddt_i}{|t_i|_{}^{d_i + \alpha}}.
$$
Also other Minkowski distances $|x|_{d_i,p} := \left(\sum_{j=1}^{d_i} |x_j|^p\right)^{1/p}$, for $p \in (1,2]$ can be written in the form \eqref{eq:cndf}; see \cite[Lemma~2.2 and Table~1]{part1} for this and further examples.

For the following results and proofs it will be useful to introduce some notation for various distributional copies of the vector $\bm{X} = (X_1, \dots, X_n)$. Recall that $\mathcal{L}(X_i)$ denotes the law of $X_i$ and define the random vectors
\begin{equation} \label{eq:X0X1}
    \begin{aligned}
    \bm{X}_0    &= (X_{0,1},\dots,X_{0,n}) &&\sim \quad\mathcal{L}(X_1) \otimes \dots \otimes \mathcal{L}(X_n),\\
    \bm{X'}_0   &= (X'_{0,1},\dots,X'_{0,n}) &&\sim \quad\mathcal{L}(X_1) \otimes \dots \otimes \mathcal{L}(X_n),\\
    \bm{X}_1    &= (X_{1,1},\dots,X_{1,n}) &&\sim \quad\mathcal{L}(X_1, \dots, X_n),\\
    \bm{X'}_1   &= (X'_{1,1},\dots,X'_{1,n}) &&\sim \quad\mathcal{L}(X_1, \dots, X_n),
    \end{aligned}
\end{equation}
such that the random vectors $\bm{X}_0, \bm{X'}_0, \bm{X}_1, \bm{X'}_1$ are independent. Note that the subscript `$1$' -- as in $\bm{X}_1$ and $\bm{X'}_1$  -- indicates that these vectors have the \emph{same distribution} as $\bm{X}$, while the subscript `$0$' -- as in $\bm{X}_0$ and $\bm{X'}_0$ -- means that these random vectors have the \emph{same marginal distributions} as $\bm{X}$, but their coordinates are independent.
\begin{definition} \label{def:moment}
We introduce the following moment conditions:

    \smallskip\textup{a)}
    The \emph{mixed moment condition} holds if
    $$
        \Ee\left( \prod_{i=1}^n \psi_i(X_{k_i,i}-X'_{l_i,i})\right)< \infty
        \quad \text{for all\ \ } k_i,l_i \in\{0,1\},\; i=1,\dots, n.
    $$

    \smallskip\textup{b)}
    The \emph{psi-moment condition} holds if there exist $p_i \in [1,\infty)$ satisfying $\sum_{i=1}^n p_i^{-1} = 1$ such that
    $$
        \Ee \psi_i^{p_i}(X_i) < \infty
        \quad\text{for all\ \ } i = 1, \dots, n.
    $$
    In particular, one may choose $p_1 = \dots = p_n = n$. (The case $p_i=\infty$ is also admissible, but this means that $\psi_i$ must be bounded or $X_i$ must have compact support.)

    \smallskip\textup{c)}
    The \emph{$2p$-moment condition} holds if there exist $p_i \in [1,\infty)$ satisfying $\sum_{i=1}^n p_i^{-1} = 1$ such that
    $$
        \Ee \big[|X_i|_{}^{2 p_i}\big] < \infty \quad \text{for all\ \ } i = 1, \dots, n;
    $$
    (the case $p_i=\infty$ is also admissible, but this means that $X_i$  is a.s.\ bounded).
\end{definition}
As shown in Lemma~\ref{lem:moments} in the supplement \cite{part2supp}, these moment conditions are ordered from weak to strong, i.e.~c) implies b) and b) implies a). Also note that b) and a) trivially hold (for any choice of $p_i$) if the functions $\psi_i$ are bounded.

\section{Distance multivariance and total distance multivariance}\label{sec:theory}

\subsection{Total distance multivariance characterizes independence}

We need the concept of $m$-independence of $n\geq m$ random variables.

\begin{definition}
    Random variables $X_1,\dots,X_n$ are \emph{$m$-independent} (for some $m\leq n$) if for any sub-family $\{i_1,\dots, i_m\}\subset \{1,\dots, n\}$ the random variables $X_{i_1},\dots, X_{i_m}$ are independent.
\end{definition}

The condition of $(n-1)$-independence allows certain factorizations of expectations of products; the proof of the following Lemma is given in the supplement \cite{part2supp}:
\begin{lemma}\label{lem:factorization}
    Let $Z_1,\dots,Z_n$ be $\comp$-valued random variables which are $(n-1)$-in\-de\-pen\-dent. Then
    \begin{equation}
        \Ee\left(\prod_{i=1}^n (Z_i - \Ee Z_i)\right) = \Ee\left(\prod_{i=1}^n Z_i - \prod_{i=1}^n \Ee Z_i\right).
    \end{equation}
\end{lemma}
If we use the random variables $Z_i := \ee^{\ii \scalp{X_i}{t_i}}$, Lemma \ref{lem:factorization} yields the following result for characteristic functions.
\begin{corollary}\label{cor:mInd-charfun}
    Let $X_1,\dots,X_m$ be $(m-1)$-independent random variables, then
    \begin{equation}\begin{aligned}
        &\Ee\left[\prod_{k=1}^m \left(\ee^{\ii \scalp{X_{i_k}}{t_{i_k}}} - f_{X_{i_k}}(t_{i_k})\right)\right]\\
        &\qquad= f_{(X_{i_1},\dots,X_{i_m})}(t_{i_1},\dots,t_{i_m}) - f_{X_{i_1}}(t_{i_1})\cdot \ldots \cdot f_{X_{i_{m}}}(t_{i_m}).
    \end{aligned}\end{equation}
\end{corollary}

This enables us to show that independence is indeed characterized by total distance multivariance.
\begin{theorem}\label{thm:independence}
\textup{a)}
    Distance multivariance vanishes for independent random variables, i.e.
    \begin{equation} \label{eq:indep-M}
        X_1,\dots,X_n \text{\ \ are independent} \implies M_\rho(X_1,\dots,X_n) = 0.
    \end{equation}
    If $X_1, \dots, X_n$ are $(n-1)$-independent, then also the converse %implication
     holds.

    \smallskip\textup{b)}
    Total distance multivariance characterizes independence, i.e.
    \begin{equation}
        X_1,\dots,X_n \text{\ \ are independent} \iff \overline{M}_\rho(X_1,\dots,X_n) = 0.
    \end{equation}
\end{theorem}

\begin{remark} \label{rem:M-n-independence} Note that multivariance is not just a building block of total multivariance, but has applications in its own right. The characterization of $n$-independence by $(n-1)$-independence and $M_\rho(X_1,\ldots,X_n) =0$ can be used to detect (higher order) dependence structures; this is used in \cite{Boet2017}. Other applications can be found in the setting of independent component analysis (ICA). The algorithm of \cite{Como1994}) aims to transform the input signal into pairwise independent random variables which, if all assumptions of ICA are satisfied, are also mutually independent. Thus, distance multivariance can be used to test the validity of assumptions by testing for higher order dependence, given pairwise independence \cite{Bart2018}.
\end{remark}

\begin{proof}[Proof of Theorem \ref{thm:independence}]
Suppose that $X_1,\dots,X_n$ are independent. We have for all indices $\{i_1,\dots,i_m\}\subset\{1,\dots,n\}$
\begin{equation}
    \Ee\left[ \prod_{k=1}^m \left(\ee^{\ii \scalp{X_{i_k}}{t_{i_k}}}- f_{X_{i_k}}(t_{i_k})\right)\right]
    = \prod_{k=1}^m \Ee\left( \ee^{\ii \scalp{X_{i_k}}{t_{i_k}}}- f_{X_{i_k}}(t_{i_k})\right)
    = 0,
\end{equation}
and, so, $M_{\bigotimes_{k=1}^m \rho_{i_k}}(X_{i_1},\dots,X_{i_m}) = 0$; this implies $\overline{M}_\rho(X_1,\dots,X_n)=0$.

For the converse statements suppose first that $X_1, \dots, X_n$ are $(n-1)$-independent and consider
$$
    \kappa(t_1, \dots, t_n)
    := \Ee\left[ \prod_{i=1}^n \left(\ee^{\ii \scalp{X_{i}}{t_{i}}}- f_{X_{i}}(t_{i})\right)\right].
$$
By definition, $M_\rho(X_1, \dots, X_n)$ is the $L^2(\rho)$-norm of $\kappa$. Since $\rho$ has full topological support and $\kappa$ is continuous, $M_\rho = 0$ implies that $\kappa \equiv 0$ everywhere on $\real^d$. By Corollary~\ref{cor:mInd-charfun}, it follows that
$$
    f_{(X_{1},\dots,X_{n})}(t_1,\dots,t_n)  = f_{X_{1}}(t_1)\cdot \ldots \cdot f_{X_{n}}(t_n)
    \quad\text{for all\ \ } t_1,\dots, t_n,
$$
i.e.\ the joint characteristic function of $X_1, \dots, X_n$ factorizes, and we conclude that $X_1, \dots, X_n$ are independent.

Finally, suppose that $\overline{M}_\rho(X_1,\dots,X_n)=0$, and thus that
\begin{equation} \label{eq:proofindependence}
    M_{\bigotimes_{k=1}^m \rho_{i_k}}(X_{i_1},\dots,X_{i_m}) = 0
    \quad\text{for any\ \ } \{i_1,\dots,i_m\} \subset \{1,\dots,n\}.
\end{equation}
Starting with subsets of size $2$, we note that
\begin{align}
    \overline{M}_{\rho_{i_1}\otimes\rho_{i_2}}(X_{i_1},X_{i_2})
    &= M_{\rho_{i_1}\otimes\rho_{i_2}}(X_{i_1},X_{i_2})\\
    \notag &= \|f_{(X_{i_1},X_{i_2})} - f_{X_{i_1}}f_{X_{i_2}}\|_{L^2(\rho_{i_1}\otimes\rho_{i_2})} = 0
\end{align}
for all $\{i_1, i_2\} \subset \{1, \dots, n\}$; this means that the random variables $X_1,\dots, X_n$ are pairwise independent, hence $X_1,\dots,X_n$ are $2$-independent. Continuing with subsets of size $3$, \eqref{eq:proofindependence} together with the first part of the proof implies $3$-independence of $X_1, \dots, X_n$. Repeating this argument finally yields the independence of $X_1,\dots,X_n$.
\end{proof}

\subsection{Further properties and representations of multivariance}
Directly from Definition~\ref{def:multivariance} we see that for two random variables $X=X_1$ and $Y=X_2$ and L\'evy measures $\rho = \rho_1\otimes\rho_2$ the notions of multivariance $M_\rho$, total multivariance $\overline M_\rho$ and generalized distance covariance $V$ as defined in \cite[Def.~3.1]{part1}
coincide, i.e.
$$
    M_\rho(X,Y) = \overline M_\rho(X,Y) = V(X,Y).
$$
The following properties are straightforward.
\begin{proposition}
    Distance multivariance enjoys the following properties.
    \begin{gather}
    \label{eq:Msingle}
        M_{\rho_i}(X_i) = 0 \quad\text{for all\ \ } i=1,\dots,n,\\
    \label{eq:Msymmetric}
        M_\rho(X_1,\dots,X_n) = M_\rho(c_1 X_1,\dots,c_n X_n)\quad \text{for\ \ } c_i \in \{-1,+1\}.
    \intertext{Let $S\subset \{1,\dots,n\}$. If $(X_i, i\in S)$ is independent of $(X_i, i\in S^c)$, then}
    \label{eq:Mfactors}
        M_\rho(X_1,\dots,X_n) = M_{\bigotimes_{i\in S}\rho_i}(X_i, i\in S) \cdot M_{\bigotimes_{i\in S^c}\rho_i}(X_i, i\in S^c).
    \end{gather}
\end{proposition}
\begin{proof}
If $n=1$, the expectation in \eqref{def:multivariance} becomes $\Ee\left(e^{\ii X_i t_i} - \Ee\ee^{\ii X_i t_i}\right)  = 0$ and \eqref{eq:Msingle} follows. Property~\eqref{eq:Msymmetric} follows from the symmetry of the measures $\rho_i$. For the last property, note that the assumption of independence allows us to factorize the following expression
\begin{align*}
    &\Ee\left[\bigotimes_{i=1}^n \left(\ee^{\ii \scalp{X_i}\sbullet}- f_{X_i}(\sbullet)\right)\right] \\
    &\qquad= \Ee\left[\bigotimes_{i \in S} \left(\ee^{\ii \scalp{X_i}\sbullet}- f_{X_i}(\sbullet)\right)\right]
    \cdot \Ee\left[ \bigotimes_{i \in S^c} \left(\ee^{\ii \scalp{X_i}\sbullet}- f_{X_i}(\sbullet)\right)\right].
\end{align*}
Since also $\rho$ can be factorized into $\bigotimes_{i \in S} \rho_i$ and $\bigotimes_{i \in S^c} \rho_i$, \eqref{eq:Mfactors} follows.
\end{proof}

Another relevant aspect is the behaviour of (total) distance multivariance, when an independent component is added to a given random vector.
\begin{proposition}Let $X_{n+1}$ be independent from $(X_1, \dots, X_n)$. Then
\begin{align}
    M_\rho(X_1,\dots,X_{n+1}) &= 0 \label{eq:M_add}\\
    \overline{M}_\rho(X_1,\dots,X_{n+1}) &= \overline{M}_\rho(X_1,\dots,X_n) \label{eq:M_total_add}
\end{align}
\end{proposition}
\begin{proof}
    The first equation follows from \eqref{eq:Mfactors} by taking $S = \{1, \dots, n\}$. If we insert this into \eqref{def:total_multivariance}, we see that all summands containing the index $i = n+1$ do not contribute to total distance multivariance. Hence, \eqref{eq:M_total_add} follows.
\end{proof}

\begin{remark}
    In this context, it is interesting to anticipate \emph{normalized} total distance multivariance $\overline{\Mskript}_\rho$ which will be defined in \eqref{eq:norm_total_DM}. If $X_{n+1}$ is independent from $(X_1, \dots, X_n)$ it is easy to check that
    \begin{gather*}
        \overline{\Mskript}_\rho(X_1, \dots, X_{n+1}) = r(n) \cdot \overline{\Mskript}_\rho(X_1, \dots, X_n)
    \end{gather*}
    where $r(n) = \sqrt{(2^n - n - 1)}/\sqrt{(2^{n+1} - n - 2)}$. Note that $r(n)$ is strictly increasing from $r(2) = 1/2$ to $\lim_{n \to \infty} r(n) = 1/\sqrt{2}$. Thus, the addition of an independent component affects $\overline{\Mskript}_\rho$ by a factor from $[1/2, 1/\sqrt{2})$.
\end{remark}

We now turn to different representations of multivariance. The representation as $L^2(\rho)$-norm in \eqref{def:multivariance} is always well-defined, but may have infinite value. Under suitable moment conditions, multivariance is finite and can be represented in terms of the continuous negative definite functions $\psi_i$ given in \eqref{eq:cndf}. The proof of the following proposition can be found in the supplement \cite{part2supp}.
\begin{proposition}\label{prop:representation}
    Multivariance  $M_\rho = M_\rho^2(X_1,\dots,X_n)$  can be written as
    \begin{gather}\label{eq:Msum_1}
        M_\rho^2 %(X_1,\dots,X_n)
        = \int \Ee\left(\sum_{k,l \in \{0,1\}^{ n}} \sgn(k,l) \prod_{i=1}^n \ee^{\ii \scalp{(X_{k_i,i}-X'_{l_i,i})}{t_i}}\right) \rho(\ddt),
    \intertext{or}\label{eq:Msum_2}
        M_\rho^2 %(X_1,\dots,X_n)
        = \int \Ee\left(\sum_{k,l \in \{0,1\}^{ n}} \sgn(k,l) \prod_{i=1}^n \left[\cos(\scalp{(X_{k_i,i}-X'_{l_i,i})}{t_i}) - 1\right]\right) \rho(\ddt),
    \end{gather}
    where
\begin{align*}
        \sgn(k,l)
        := (-1)^{\sum\limits_{j=1}^n (k_j+l_j)}
        =
        \begin{cases}
            +1, & \text{if $(k,l)$ contains an even no.\ of `$1$\!'\/s},\\
            -1, & \text{if $(k,l)$ contains an odd no.\ of `$1$\!'\/s}.
        \end{cases}
\end{align*}

    If one of the moment conditions in Definition~\ref{def:moment} holds, then the distance multivariance $M_\rho(X_1, \dots, X_n)$ is finite, and the following representation holds
    \begin{equation}\label{eq:Mprod}\begin{aligned}
        M_\rho^2%(X_1,\dots,X_n)
        &=\Ee\Bigg(\prod_{i=1}^n \big[-\psi_i(X_{i}-X'_{i})+\Ee(\psi_i(X_{i}-X'_{i}) \mid X_i)\\
        &\qquad\qquad\quad\mbox{} + \Ee(\psi_i(X_{i}-X'_{i}) \mid X'_i) -\Ee\psi_i(X_{i}-X'_{i})\big]\Bigg).
    \end{aligned}\end{equation}
\end{proposition}

\begin{remark}
\ \ a) \
The representations \eqref{eq:Msum_1} and  \eqref{eq:Msum_2} have an interesting structural resemblance to the Leibniz' formula for determinants; %. The representation
\eqref{eq:Mprod} is the analogue of \cite[Cor.~3.5]{part1}
for the bivariate case.

\smallskip b) \
In the bivariate case $n=2$, distance multivariance is also finite under the weaker moment condition $\Ee\psi_1(X_1) + \Ee\psi_2(X_2) < \infty$, cf.~\cite[Thm.~3.7]{part1}.
\end{remark}

We introduce yet another representation of distance multivariance, which helps to clarify the relation to the finite-sample form and the representation as \emph{Gaussian multivariance}, given in Section~\ref{sub:gaussian} below. For this, we need the centering operator $\Ce_\Fskript$:
\begin{proposition}\label{prop:center}
    Let $X$ be an integrable random variable on $(\Omega,\Askript,\Pp)$ and $\Fskript, \Fskript'$ be sub-$\sigma$-algebras of $\Askript$. Set
    \begin{equation}
        \Ce_\Fskript X:= X- \Ee(X\mid\Fskript).
    \end{equation}
    Then $\Ce$ is a linear operator and
    \begin{align}
        \Ce_{\{\emptyset,\Omega\}} X
        &= X- \Ee X,\\
        \label{perp-double}\Ce_\Fskript \Ce_{\Fskript'} X
        &= X - \Ee(X \mid \Fskript') - \Ee(X \mid \Fskript) + \Ee(\Ee(X \mid \Fskript') \mid \Fskript),\\
        \label{perp-meas}
        \Ce_\Fskript \Ce_{\Fskript'} X &= 0 \quad \text{if $X$ is $\Fskript'$-measurable.}
    \end{align}
    If $\Fskript'$ and $\Fskript$ are independent, then $\Ee(\Ce_{\Fskript'}X \mid\Fskript) = \Ce_{\{\emptyset,\Omega\}}\Ee(X \mid \Fskript)$.
\end{proposition}
All assertions of the proposition follow directly from the properties of conditional expectations, and we omit the proof. Geometrically, $\Ce_\Fskript X$ can be interpreted as the residual from the orthogonal projection of $X$ onto the set of $\Fskript$-measurable functions. We will use the shorthand $\Ce_X := \Ce_{\sigma(X)}$.

\begin{corollary}\label{cor:M_center}
If one of the moment conditions in Definition~\ref{def:moment} holds, then
\begin{equation}\label{eq:Mprod_center}
    M_\rho^2(X_1,\dots,X_n)
    = \Ee\left(\prod_{i=1}^n - \Ce_{X_i} \Ce_{X'_i} \psi_i(X_i - X'_i)\right)
\end{equation}
and
\begin{equation}\label{eq:Mprod_center_total}
    \overline{M}_\rho^2(X_1,\dots,X_n)
    = \Ee\left(\prod_{i=1}^n \left(1 - \Ce_{X_i} \Ce_{X'_i} \psi_i(X_i - X'_i)\right)\right) - 1.
\end{equation}
The factors can be written explicitly as
\begin{align}\label{eq:psi_double_center}
    \Ce_{X_i} \Ce_{X'_i} \psi_i(X_i - X'_i) =\,& \psi_i(X_i - X'_i) - \Ee[\psi_i(X_i - X'_i) \mid X'_i] \\ &- \Ee[\psi_i(X_i - X'_i) \mid X_i] + \Ee\psi_i(X_i - X'_i). \notag
\end{align}
\end{corollary}

\begin{proof}
The identity \eqref{eq:psi_double_center} follows directly from the definition of the double centering operator in Prop.~\ref{prop:center}. The representation~\eqref{eq:Mprod_center} is an immediate consequence of \eqref{eq:Mprod} in Prop.~\ref{prop:representation}.
For representation~\eqref{eq:Mprod_center_total} of the total multivariance, write $a_i := - \Ce_{X_i} \Ce_{X'_i} \psi_i(X_i - X'_i)$. We can expand the product
$$
    \prod_{i=1}^n (1 + a_i)
    = \sum_{m=0}^n e_m(a_1, \dots, a_n),
$$
where the function $e_m(a_1, \dots, a_n)$ is the $m$th elementary symmetric polynomial in $(a_1, \dots, a_n),$ i.e.
$$e_m(a_1, \dots, a_n) = \sum_{1\leq i_1< \dots < i_m\leq n} a_{i_1} \cdot \ldots \cdot a_{i_m}.$$
In particular, $e_0(a_1, \dots, a_n) = 1$ and $e_1(a_1, \dots, a_n) = a_1 + \dots + a_n$. Taking expectations yields
\begin{equation}\begin{aligned}
    \Ee \left[\prod_{i=1}^n (1+ a_i)\right] - 1  &= \sum_{m=1}^n \Ee\left[ e_m(a_1, \dots, a_n) \right] - 1 \\
    &= \sum_{\substack{1\leq i_1< \dots < i_m \leq n\\2 \leq m \leq n}} \Ee\left[ a_{i_1} \cdot \ldots \cdot a_{i_m} \right] \\
    &= \sum_{\substack{1\leq i_1< \dots < i_m \leq n\\2 \leq m \leq n}} M^2_\rho(X_{i_1},\dots,X_{i_m}) \\&= \overline{M}_\rho^2(X_1,\dots,X_n),
\end{aligned}\end{equation}
as claimed. Note that the first elementary symmetric polynomial $e_1$ does not contribute since $\Ee[a_i] = 0$ for all $i \in \{1, \dots, n\}$.
\end{proof}

\subsection{Gaussian multivariance}\label{sub:gaussian}
Recall that for a real-valued negative definite function $\psi:\real^d\to\real$ the matrix $(\psi(\xi_j)+\psi(\xi_k)-\psi(\xi_j-\xi_k))_{j,k=1,\dots,n}$, $\xi_1,\dots,\xi_n\in\real^d$, $n\in\nat$, is positive semidefinite, see \cite[Def.~3.6.6]{Jaco2001}.
Therefore, we can associate with any cndf $\psi$ some Gaussian random field indexed by $\real^d$.
\begin{definition}\label{def:gaussian}
    Assume that %the random variables
    $X_1, \dots, X_n$ satisfy one of the moment conditions in Definition~\ref{def:moment} and let $G_1,\dots, G_n$ be independent (also independent of $X_1, \dots, X_n$), stationary Gaussian random fields with
    \begin{equation}\label{gaussianprocess}
        \Ee G_i(\xi) = 0
        \quad\text{and}\quad
        \Ee(G_i(\xi)G_i(\eta)) = \psi_i(\xi) + \psi_i(\eta) - \psi_i(\xi-\eta)
    \end{equation}
    for $\xi,\eta \in \real^{d_i}$.
    The \emph{Gaussian multivariance} of $(X_1, \dots, X_n)$ is defined by
    \begin{equation}
        \GC^2(X_1,\dots,X_n) = \Ee\left(\prod_{i=1}^n X_i^{G_i}X_i'^{G_i}\right)
    \end{equation}
    where $(X_1',\dots, X_n')$ is an independent copy of $(X_1,\dots,X_n)$ and
    \begin{equation}\label{def:gaussiancentering}
        X_i^{G_i} := G_i(X_i) - \Ee(G_i(X_i) \mid G_i).
    \end{equation}

\end{definition}

\begin{remark}
\ \ a) \
Using the centering operator $\Ce$ from Proposition~\ref{prop:center}, we can write \eqref{def:gaussiancentering} as $X_i^{G_i} = \Ce_{G_i} G_i(X_i)$.

\smallskip b) \
In the bivariate case $n=2$ Gaussian multivariance coincides with the Gaussian covariance defined in
\cite[Sec.~7]{part1}.

\smallskip c) \
If $\psi_i$ is given by the Euclidean norm, then $G_i$ is a Brownian field indexed by $\real^{d_i}$. In particular, if $n=2$ and both $\psi_1$ and $\psi_2$ are given by the Euclidean norm, then $\GC(X_1,X_2)$ coincides with the \emph{Brownian covariance} of Sz\'ekely and Rizzo \cite{SzekRizz2009}.

\smallskip d) \
    If $\psi_i(x) = |x|^\alpha$, then $G_i$ is a fractional Brownian field with Hurst exponent $H = \frac{\alpha}{2}$, cf.~\cite[Sec.~4]{SzekRizz2009}.
\end{remark}

\begin{theorem}\label{thm:MequalG}
Suppose that one of the moment conditions of Definition~\ref{def:moment}  holds and $\Ee(\psi_i(X_i)^{\frac{n}{2}}) < \infty$ for $i=1,\ldots,n.$ Then distance multivariance and Gaussian multivariance coincide, i.e.
\begin{equation}
    M_\rho(X_1,\dots,X_n) = \GC(X_1,\dots,X_n).
\end{equation}
\end{theorem}
\begin{proof}
By Corollary~\ref{cor:M_center} we can represent squared multivariance in the product form \eqref{eq:Mprod_center}. Each of the factors can be rewritten as
    \begin{equation}
    \begin{aligned}
    &- C_X C_{X'} \psi(X-X')  \\
    &\quad=  C_X C_{X'}\left(\psi(X)+\psi(X') -\psi(X-X')\right)\\
    &\quad= C_X C_{X'} \Ee(G(X)G(X') \mid X,X') \\
    &\quad= \Ee(G(X)G(X') \mid X,X') - \Ee(G(X) G(X') \mid X)\\
    &\quad\qquad\qquad\mbox{} - \Ee(G(X)G(X') \mid X') + \Ee(G(X)G(X'))\\
    &\quad= \Ee\left[ \left( G(X)- \Ee(G(X) \mid G)\right) \left(G(X')- \Ee(G(X') \mid G)\right) \:\middle|\: X,X'\right]\\
    &\quad  = \Ee(X^G X'^G \mid X,X'),\\
    \end{aligned}
    \end{equation}
where we have used the covariance structure \eqref{gaussianprocess} of the Gaussian process $G$ in the third line.
Putting everything together, we have
\begin{align*}
     &M_\rho^2(X_1,\dots,X_n)
    = \Ee \left(\prod_{i =1}^n - C_{X_i} C_{X'_i} \psi_i(X_i-X_i')\right) \\
    &= \Ee\left(\prod_{i=1}^n \Ee\left( X_i^{G_i}X_i'^{G_i} \;\middle|\; X_i,X_i'\right)\right)=  \Ee\left(\prod_{i=1}^n X_i^{G_i}X_i'^{G_i}\right)
        = \GC^2(X_1,\dots,X_n).
\end{align*}
Note that for the penultimate equality the absolute integrability of the integrand, i.e. $\Ee\left(\prod_{i=1}^n |X_i^{G_i}X_i'^{G_i}|\right)< \infty$, is required.

Writing $\Fskript := \sigma(X_i,i=1,\ldots,n)$ and $\Fskript' := \sigma(X_i',i=1,\ldots,n)$, we obtain
\begin{equation*}
\begin{aligned}
\Ee&\left(\prod_{i=1}^n |X_i^{G_i}X_i'^{G_i} |\right) = \Ee\left(\prod_{i=1}^n \Ee\left(|X_i^{G_i}X_i'^{G_i} | \;\middle|\; \Fskript,\Fskript' \right)\right)\\
&\leq \Ee\left(\prod_{i=1}^n \sqrt{\Ee\left(|X_i^{G_i}|^2 \;\middle|\; \Fskript,\Fskript'\right)\Ee\left(|X_i'^{G_i} |^2 \;\middle|\; \Fskript,\Fskript'\right)}\right)\\
&= \Ee\left(\sqrt{\prod_{i=1}^n \Ee\left(|X_i^{G_i}|^2 \;\middle|\;\Fskript \right)}\right) \cdot \Ee\left(\sqrt{\prod_{i=1}^n \Ee\left(|X_i'^{G_i}|^2 \;\middle|\; \Fskript'\right)}\right)\\
\label{eq:jensenm}& = \Ee\left(\sqrt{\prod_{i=1}^n \Ee\left(|X_i^{G_i}|^2 \;\middle|\; \Fskript\right)}\right)^2
\leq \left(\prod_{i=1}^n\Ee\left[\left(\Ee\left(|X_i^{G_i}|^2 \;\middle|\;\Fskript\right)\right)^\frac{n}{2}\right]\right)^\frac{2}{n}\\
&\leq \left(\prod_{i=1}^n\Ee\left[\Ee\left(|X_i^{G_i}|^n \;\middle|\;\Fskript\right)\right]\right)^\frac{2}{n}
= \left(\prod_{i=1}^n\Ee\left(|X_i^{G_i}|^n \right)\right)^\frac{2}{n},
\end{aligned}
\end{equation*}
where we used successively the independence of the $G_i$, the conditional H\"older inequality \cite[7.2.4]{ChowTeic1997}, the independence and identical distribution of $(X_i, i=1,\ldots,n)$ and $(X_i', i=1,\ldots,n)$, the generalized H\"older inequality \cite[p. 133, Pr. 13.5]{schilling-mims} and the conditional Jensen inequality \cite[7.1.4]{ChowTeic1997}.\\
Finally, note that for $n\in\nat$ the elementary inequality $|a+b|^n \leq 2^{n-1} (|a|^n+|b|^n)$ and the formula for absolute moments of Gaussian random variables, i.e. $\Ee(|G_i(t)|^n) = 2^\frac{n}{2} \Gamma(\tfrac{n+1}{2}) \pi^{-\frac{1}{2}} [\Ee G_i(t)^2]^\frac{n}{2},$ and $\Ee [G_i(t)^2] = 2\psi_i(t)$ imply
\begin{equation}
\Ee|X_i^{G_i}|^n \leq 2^n \Ee|G_i(X_i)|^n
 = 2^{2n} \Gamma(\tfrac{n+1}{2})\pi^{-\frac{1}{2}} \Ee(\psi_i(X_i)^{\frac{n}{2}}).
\end{equation}
which proves the desired integrability.
\end{proof}

We conclude this section by comparing (total) distance multivariance to related approaches in \cite{BakiSzek2011} and to the multivariate Hilbert-Schmidt independence criterion ($\mathrm{dHSIC}$) of \cite{PfisBuehSchoPete2017}.

\subsection{Comparison with \textup{\cite{BakiSzek2011}}} \label{sec:BS11}
The problem of generalizing distance covariance of two random variables $X,Y$ to multiple variables has been discussed in a short paragraph `\emph{How to \textup{(}not\textup{)} extend \textup{[}distance covariance\textup{]} $\mathcal{V}(X,Y)$ to more than two random variables}' in \cite{BakiSzek2011}. In the notation of our paper they discuss for three random variables $X,Y,Z$ the following objects:

    \medskip a)
    Gaussian Covariance $\GC(X,Y,Z) = \Ee\left(X^G X'^G Y^G Y'^G Z^G Z'^G\right)$ (cf.~Section~\ref{sub:gaussian}) where $G$ is a Brownian motion. This approach is dismissed in \cite{BakiSzek2011} since it does not characterize the independence of $X,Y,Z$.

    \smallskip b)
    The quantity
    \begin{equation}\label{eq:BS_int}
        \int_{\real^d} \left|\Ee \left[\ee^{\ii (\scalp{X}{t_1} + \scalp{Y}{t_2} + \scalp{Z}{t_3})}\right] - f_{X}(t_1) f_{Y}(t_2) f_{Z}(t_3)\right|^2 \rho(\ddt_1, \ddt_2, \ddt_3);
    \end{equation}
    -- this should be compared with the similar, yet different expression \eqref{eq:M_rho_L2}. Bakirov and Sz\'ekely dismiss this approach, since the integral can become infinite if $Z\equiv 0$, even if $X$ and $Y$ are bounded and independent; note that in this case the three random variables $X,Y,Z$ are actually independent.

    \smallskip c)
    The (bivariate) distance covariance of $U \sim \mathcal{L}(X,Y,Z)$ and $V \sim \mathcal{L}(X) \otimes \mathcal{L}(Y) \otimes \mathcal{L}(Z)$. Bakirov and Sz\'ekely recommend to use this approach, since it is able to detect independence of $X,Y,Z$, but they do not follow up this approach with a deeper discussion.

\medskip
Comparing with our results, let us add a few comments. The approach a) is equivalent to the calculation of distance multivariance $M_\rho(X,Y,Z)$ (based on Euclidean distance), by Theorem~\ref{thm:MequalG}. Consistent with the remarks of \cite{BakiSzek2011}, distance multivariance cannot characterize independence, cf.~Theorem~\ref{thm:independence}. It serves, however, as a building block of \emph{total distance multivariance}, which \emph{does} characterize independence.

If $Z\equiv 0$, the expression \eqref{def:multivariance} is zero, i.e.\ it does not suffer from the particular integrability problems as \eqref{eq:BS_int}. However, under certain conditions, it coincides with \eqref{eq:BS_int}, see Corollary~\ref{cor:mInd-charfun}.

Compared with c), our approach has the advantage that both distance multivariance and total distance multivariance have a very simple and efficient finite-sample representation, which retains all the benefits of the bivariate distance covariance, cf.~Theorem~\ref{thm:sample}. Also the asymptotic properties of the estimators are similar to the bivariate case, cf.~Theorems~\ref{thm:Mdistconv}, \ref{thm:Mdconv} and Section~4
in \cite{part1}.

\subsection{Comparison with $\mathrm{dHSIC}$} \label{sec:dhsic}
The multivariate Hilbert-Schmidt independence criterion ($\mathrm{dHSIC}$) was recently introduced in \cite{PfisBuehSchoPete2017}. Using our notation, $\mathrm{dHSIC}$ is given by
\begin{equation}\label{eq:dHSIC}\begin{aligned}
    \mathrm{dHSIC}(X_1,\ldots,X_n)
    &:= \Ee\left[ \prod_{i=1}^n k_i(X_i,X_i')\right] + \prod_{i=1}^n \Ee\left[k_i(X_i,X_i')\right] \\
    &\qquad\mbox{}- 2 \Ee\left[\prod_{i=1}^n \Ee\left[ k_i(X_i,X_i') \mid X_i\right]\right],
\end{aligned}\end{equation}
where the $k_i$ are continuous, bounded, characteristic, positive semidefinite kernels on $\real^{d_i}$. Here, a kernel $k(x,y)$ is said to be characteristic, if
\begin{gather*}
    \mu \mapsto \Pi(\mu) = \int k(x,\cdot) \mu(dx)
\end{gather*}
from the finite Borel measures to a suitable Hilbert space is an injective map, see \cite[Section 2.1]{PfisBuehSchoPete2017}) for details.

Note that any continuous negative definite function $\psi_i$ gives rise to a continuous positive semidefinite kernel under the correspondence
\begin{equation}\label{eq:psi_to_kernel}
    k_i(x,y) = \psi_i(x) + \psi_i(y) - \psi_i(y-x),
\end{equation}
see \cite{SejdSripGretFuku2013}. In the bivariate case ($n=2$) it is shown in \cite{SejdSripGretFuku2013} that $\mathrm{dHSIC}$ is equivalent to distance covariance with (quasi-)distance $\psi_i$. This raises the question whether equivalence of $\mathrm{dHSIC}$ and (total) distance multivariance still holds in the case $n > 2$. It can be easily shown by numerical experiments that they are not identical, at least not under the correspondence \eqref{eq:psi_to_kernel}. Nevertheless, the experiments show a strong positive association between $\mathrm{dHSIC}$ and total multivariance. Clarifying the exact nature of this association remains an open question, but we present the following related result: Given the marginal distributions $\mathcal{L}(X_1), \dots, \mathcal{L}(X_n)$, we can find kernels $k_i$, \emph{depending on these distributions}, such that $\mathrm{dHSIC}$ coincides formally with (total) distance multivariance on the random vector $(X_1, \dots, X_n)$. Note that, in general, these kernels are unbounded and its sample versions depend on all samples, thus they are beyond the restrictions imposed in \cite{PfisBuehSchoPete2017}.

\begin{proposition}
    Let $X_1, \dots, X_n$ satisfy one of the moment conditions of Definition~\ref{def:moment} and define the kernels
    \begin{equation}\label{eq:kernel_lambda}\begin{aligned}
        k^\lambda_i(x_i,x_i')
        &:= -\psi_i(x_i-x_i') + \Ee(\psi_i(x_i - X_i')) \\
        &\qquad\mbox{}+ \Ee(\psi_i(X_i - x_i')) - \Ee(\psi_i(X_i-X_i')) + \lambda,
    \end{aligned}\end{equation}
    where $\lambda \ge 0$ and write $\mathrm{dHSIC}^\lambda$ for the corresponding quantity defined in \eqref{eq:dHSIC}. Then
    \begin{align}
        \mathrm{dHSIC}^0(X_1, \dots, X_n) &= M_\rho^2(X_1, \dots, X_n),\\
        \mathrm{dHSIC}^1(X_1, \dots, X_n) &= \overline M_\rho^2(X_1,\dots,X_n).
    \end{align}
    The kernel $k^0_i$ is not characteristic in the sense of \textup{\cite[Section 2.1]{PfisBuehSchoPete2017}}.
\end{proposition}
\begin{proof}
    Observe that $\Ee \left[k^\lambda_i(X_i,X_i')\right] = \Ee\left[ k^\lambda_i(X_i,X_i') \mid X_i\right]  = \lambda$, such that \eqref{eq:dHSIC} simplifies to
    \begin{gather*}
        \mathrm{dHSIC}^\lambda(X_1,\ldots,X_n) := \Ee\left[ \prod_{i=1}^n \left(\lambda + k^0_i(X_i,X_i')\right)\right] - \lambda^n.
    \end{gather*}
    This is equal to \eqref{eq:Mprod_center} for $\lambda = 0$ and to \eqref{eq:Mprod_center_total} for $\lambda = 1$. It remains to show that $k_i^0$ is not characteristic. To this end, denote by $\mu_i$ the distribution of $X_i$. Then
    \begin{gather*}
        \Pi(\mu_i)(y) = \int k^\lambda(x,y) \mu_i(x) = \Ee \left[k^\lambda_i(X_i,y)\right] = \lambda.
    \end{gather*}
    If $\lambda = 0$, then $\Pi(\mu_i) = 0 = \Pi(\bm{0})$, where $\bm{0}$ is the measure of mass zero. This shows that $\Pi$ is not injective, and therefore that $k_i^0$ is not characteristic.
\end{proof}

\section{Statistical properties of distance multivariance}\label{sec:stats}

\subsection{Sample distance multivariance} \label{sec:sampleM}
We now consider a sample of $N$ observations $(\bm{x}^{(1)}, \dots, \bm{x}^{(N)})$  of the random vector $\bm{X} = (X_1, \dots, X_n)$. Every observation $\bm{x}^{(j)}$ is a vector in $\real^d$, $d= d_1 + \dots + d_n$, of the form $\bm{x}^{(j)} = \left(x_1^{(j)}, \dots, x_n^{(j)} \right)$, with each $x_i^{(j)}$ in $\real^{d_i}$. Given such a sample, we denote by $(\hat X^{\sampleN}_1,\dots,\hat X^{\sampleN}_n)$ the random vector with the corresponding empirical distribution. Evaluating distance multivariance at this vector, we obtain the sample distance multivariance
\begin{gather*}
    \MN^2_\rho(x^{(1)},\dots,x^{(N)})
    := M^2_\rho(\hat X_1^{\sampleN},\dots,\hat X_n^{\sampleN}),
\end{gather*}
which turns out to have a surprisingly simple representation.

Recall that the Hadamard (or Schur) product of two matrices $A,B\in\real^{N\times N}$ is the $N\times N$-matrix $A\circ B$ with entries $(A\circ B)_{jk} = A_{jk}B_{jk}$.
\begin{theorem}\label{thm:sample}
    Let $(\bm{x}^{(1)}, \dots, \bm{x}^{(N)})$ be a sample of size $N$.

    \smallskip\textup{a)}
    The sample distance multivariance can be written as
    \begin{equation} \label{eq:empGC}\begin{aligned}
        \MN^2_\rho(\bm{x}^{(1)},\dots, \bm{x}^{(N)})
        &= \frac{1}{N^2} \sum_{j,k=1}^N (A_1\circ \ldots \circ A_n)_{jk}\\
        &= \frac{1}{N^2} \sum_{j,k=1}^N (A_1)_{jk} \cdot \ldots \cdot (A_n)_{jk};
    \end{aligned}\end{equation}
    here, $A_i := -CB_iC$ where $B_i = \left(\psi_i\big(x_i^{(j)}-x_i^{(k)}\big)\right)_{j,k = 1,\dots,N}$ is the distance matrix and $C = I - \tfrac{1}{N}\One$ the centering matrix.

    \smallskip\textup{b)}
    The sample total distance multivariance can be written as
    \begin{equation} \label{eq:empTGC}
        \MNo^2_\rho(\bm{x}^{(1)},\dots, \bm{x}^{(N)}) =  \left[\frac{1}{N^2} \sum_{j,k=1}^N (1+(A_1)_{jk}) \cdot \ldots \cdot (1+ (A_n)_{jk})\right] - 1.
    \end{equation}
\end{theorem}

\begin{remark}\label{rem:sample}
\ \ a) \
If $n$ is even, then $A_i$ can be replaced by $-A_i.$
This explains the different sign used in the case $n=2$, cf.\ \cite[Def.~3]{SzekRizz2009}
and \cite[Lem.~4.2, Rem.~4.3]{part1}.

If $n=2$, then $\sum_{j,k=1}^N (A_1 \circ A_2)_{jk} = \trace(A_2^\top A_1)$ and the generalized sample distance covariance from \cite[Sec.~4]{part1} is recovered. If in addition $\psi_i(x) = |x|$, i.e.~the Euclidean distance, then we get the sample distance covariance of Sz\'ekely \emph{et al.} \cite{SzekRizzBaki2007, SzekRizz2009}.

\smallskip b) \
Since the $\psi_i$ are continuous negative definite functions, the matrices $-B_i$ are conditionally positive definite matrices, i.e.\ $-\lambda^\top B_i \lambda \geq 0$ for all non-zero $\lambda$ in $\real^N$ with $\lambda_1 + \dots + \lambda_N = 0$.
As the double centerings of conditionally positive definite matrices, the matrices $A_i$ are positive definite. By Schur's theorem, the $N$-fold Hadamard product of positive definite matrices is again positive definite, see Berg and Forst \cite[Lem.~3.2]{BerFor}. This gives a simple explanation as to why $\MN_\rho^2$ is always a non-negative number.

\smallskip c) \
Important special cases are when the $\psi_i$ are chosen as Euclidean distance, or as Minkowski distances. In these cases, each $B_i$ is a distance matrix. In general, $B_i$ need not be a distance matrix, since only $\sqrt{\psi_i}$, but not necessarily $\psi_i$ itself, defines a distance.  Still, $\psi_i$ always defines a quasi-metric, i.e.~a metric with a relaxed triangle inequality, cf.~\cite[Sec.~2]{part1}.

\smallskip d) \
Even though total distance multivariance is defined as the sum of the multivariances of all $2^n - 1 - n$ subfamilies of $\{X_1, \dots, X_n\}$ with at least two members, cf.~\eqref{def:total_multivariance}, its empirical version \eqref{eq:empTGC} has a computational complexity of only $\mathcal{O}(nN^2)$.

\smallskip e) \
The row- and column sums of each $A_i$ are zero. This is a consequence of the double centering $A_i = -C B_i C$. \label{item:zero}

\smallskip f) \
Equation \eqref{eq:empGC} is a direct analogue of the representation \eqref{eq:Mprod_center}, when the centering operator is replaced by the centering matrix. The same is true for \eqref{eq:empTGC} in relation to \eqref{eq:Mprod_center_total}.
\end{remark}

\begin{proof}[Proof of Theorem~\ref{thm:sample}]
Since the support of the empirical distribution is finite, the moment conditions of Definition~\ref{def:moment} are trivially satisfied. Therefore, we can use the representation \eqref{eq:Mprod_center} to get
\begin{align}
\notag
\MN^2_\rho&(x^{(1)},\dots,x^{(N)})
= M^2_\rho(\hat X_1^{\sampleN},\dots,\hat X_n^{\sampleN})\\
\notag&= \Ee\Bigg( \prod_{i=1}^n \Big[-\psi_i(\hat X^{\sampleN}_{i}-\hat X'^{\sampleN}_{i})
   + \Ee\left(\psi_i\left(\hat X^{\sampleN}_{i}-\hat X'^{\sampleN}_{i}\right) \;\middle|\; \hat X^{\sampleN}_i\right)\\
\label{eq:sample_interm}&\qquad\qquad\qquad\mbox{}
    + \Ee\left(\psi_i\left(\hat X^{\sampleN}_{i}-\hat X'^{\sampleN}_{i}\right) \;\middle|\; \hat X'^{\sampleN}_i\right)
    - \Ee \psi_i\left(\hat X^{\sampleN}_{i}-\hat X'^{\sampleN}_{i}\right)\Big]\Bigg)\\
\notag&= \frac{1}{N^2}\sum_{j,k=1}^N \Bigg( \prod_{i=1}^n \Big[  -\psi_i\left(x_i^{(j)}-x_i^{(k)}\right)
    + \Ee\left(\psi_i\left(\hat X_i^{\sampleN}-\hat X_i'^{\sampleN}\right) \;\middle|\; \hat X_i^{\sampleN} = x_i^{(j)}\right)\\
\notag&\qquad\qquad\qquad\mbox{}
    + \Ee\left(\psi_i\left(\hat X_i^{\sampleN}-\hat X_i'^{\sampleN}\right) \;\middle|\; \hat X_i'^{\sampleN} = x_i^{(k)}\right) - \Ee \psi_i\left(\hat X_i^{\sampleN}-\hat X_i'^{\sampleN}\right) \Big]\Bigg).
\end{align}
Denoting by $\One_N$ the column vector consisting of $N$ ones, we can rewrite the individual terms in \eqref{eq:sample_interm} as
\begin{subequations}\label{eq:sample_dm}
\begin{align}
    \psi_i \left(x_i^{(j)} - x_i^{(k)}\right) &= (B_i)_{jk}\\
    \Ee\left(\psi_i \left(\hat X_i - \hat X_i'\right) \;\middle|\; \hat X_i = x_i^{(j)}\right) &= \frac{1}{N} \sum_{l=1}^N (B_i)_{jl} = \frac{1}{N} \left(\One_N^\top B_i\right)_j \\ %\psi_i \left(x_i^{(j)}- x_i^{(l)}\right)
    \Ee\left(\psi_i \left(\hat X_i - \hat X_i'\right) \;\middle|\; \hat X'_i = x_i^{(k)}\right) &= \frac{1}{N} \sum_{m=1}^N (B_i)_{mk} = \frac{1}{N} \left(B_i \One_N\right)_k \\ %\psi_i \left(x_i^{(m)}- x_i^{(k)}\right)
    \Ee\psi_i \left(\hat X_i - \hat X_i'\right) &= \frac{1}{N^2} \sum_{l,m=1}^N (B_i)_{ml} = \frac{1}{N^2}\, \One_N^\top B_i \One_N. %\!\psi_i \left(x_i^{(m)}- x_i^{(l)}\right)
\end{align}
\end{subequations}
This shows that each factor on the right hand side of \eqref{eq:sample_interm} is the $(j,k)$-th entry of the matrix $A_i = -C B_i C$, and \eqref{eq:empGC} follows. The
representation \eqref{eq:empTGC} can be derived in complete analogy from \eqref{eq:Mprod_center_total}.
\end{proof}

\subsection{Estimating distance multivariance}
In this section we examine the properties of the sample distance multivariance $\MN_\rho$ as an estimator of $M_\rho$. The corresponding results for the sample total distance multivariance will be presented in the next section.

\begin{theorem}[$\MN_\rho$ is a strongly consistent estimator for $M_\rho$]\label{thm:sconsistent}
    Let one of the moment conditions of Definition~\ref{def:moment} be satisfied. Then $\MN_\rho$ is a strongly consistent estimator of $M_\rho$, i.e.
    \begin{equation}
        \MN_\rho(\bm{X}^{(1)},\dots,\bm{X}^{(N)}) \xrightarrow[N\to \infty]{} M_\rho(X_1,\dots,X_n) \quad\text{a.s.}
    \end{equation}
\end{theorem}
\begin{proof}
    Inserting the representation \eqref{eq:sample_dm} into \eqref{eq:sample_interm}, we see that $\MN_\rho$ is a $V$-statistic. Thus the convergence of the estimator $\MN_\rho$  is just the strong law of large numbers for $V$-statistics.
\end{proof}

\begin{remark}
    In the case of $n=2$ strong consistency can be obtained under the weaker moment condition $\Ee\psi_i(X_i)<\infty$ for $i=1,2$, see \cite[Thm.~4.4]{part1}.
    For $n \geq 3$ the arguments used in \cite{part1} break down. However, we show a weak consistency result under independence and relaxed moment conditions in Corollary~\ref{cor:wconsistent} below.
\end{remark}

The next result is our main result on the asymptotics properties of the estimator $\MN_\rho$. The proof is technical and relegated to the supplement \cite{part2supp}.

\begin{theorem}[Asymptotic distribution of $\MN_\rho$]\label{thm:Mdistconv}\mbox{}
    \smallskip\textup{a)}
    Let $X_1,\dots,X_n$ be independent random variables such that the moments $\Ee\psi_i(X_i)<\infty$ and $\Ee \left[\log^{1+\epsilon}(1+|X_i|^2)\right] <\infty$ exist for some $\epsilon>0$ and all $i=1,\dots, n$. Then
    \begin{equation}\label{eq:G_conv}
        N \cdot \MN_\rho^2(\bm{X}^{(1)},\dots,\bm{X}^{(N)}) \xrightarrow[N\to \infty]{d} \|\mathds{G}\|_{L^2(\rho)}^2
    \end{equation}
    where $\mathds{G}$ is a centred, i.e.~$\Ee \mathds{G}(t) =0$, $\comp$-valued Gaussian process indexed by $\real^{d}$ with  covariance function
    \begin{equation}\label{eq:G_cov}
        \Cov(\mathds{G}(t),\mathds{G}(t'))
        = \Ee\big[\mathds{G}(t)\overline{\mathds{G}(t')}\big]
        = \prod_{i=1}^n \left(f_{X_i}(t_i-t_i')-f_{X_i}(t_i) \overline{f_{X_i}(t_i')}\right).
    \end{equation}

    \smallskip\textup{b)}
    Suppose that the random variables $X_1, \dots, X_n$ are $(n-1)$-in\-de\-pen\-dent, but not $n$-in\-de\-pen\-dent and that one of the moment conditions of Definition~\ref{def:moment} holds. Then
    \begin{equation}\label{eq:Mdistconv_b}
        N \cdot \MN^2_\rho(\bm{X}^{(1)},\dots,\bm{X}^{(N)}) \xrightarrow[N\to \infty]{} \infty \quad\text{a.s.}
    \end{equation}
\end{theorem}
\begin{remark}
\ \ a) \
The complex-valued Gaussian process $\mathds{G}$ has to be distinguished from the Gaussian processes $G_i$ that appear in Definition~\ref{def:gaussian} of the Gaussian multivariance.

\smallskip b) \
Using the results of \cite{Csoe1985}, the $\log$-moment condition in a) can be relaxed by a weaker (but more involved) integral test cf.~\cite[Condition~($\star$)]{Csoe1985}.

From \cite[Lem.~2.7]{part1} it is readily seen that the log-moment condition in Thm.~\ref{thm:Mdistconv}.a) is equivalent to $\Ee\left[\log^{1+\epsilon}\left(1\vee\sqrt{|X_1|^2+\dots+|X_n|^2}\right)\right]<\infty$.

\smallskip c) \
The expectation of the limit in \eqref{eq:G_conv} can be calculated as
\begin{equation}\label{eq:EG}
    \Ee(\|\mathds{G}\|_{L^2(\rho)}^2)
    = \prod_{i=1}^n \int_{\real^{d_i}} \left(1-|f_{X_i}(t_i)|^2\right) \rho_i(d t_i)
    = \prod_{i=1}^n \Ee \psi_i(X_i-X_i').
\end{equation}

\smallskip d) \
From Lemma~\ref{lem:ZN} in the supplement \cite{part2supp} it can be seen that $\MN_\rho$ is a biased estimator of $M_\rho$, since in the case of non-degenerate and independent random variables
\begin{equation*}
    \Ee \Big[\MN^2_\rho(\bm{X}^{(1)},\dots,\bm{X}^{(n)})\Big]
    %= \Ee \|Z_N\|_\rho^2
    =  \frac{(N-1)^n+(-1)^n (N-1)}{N^{n+1}} \prod_{i=1}^n \Ee \psi_i(X_i-X_i') > 0,
\end{equation*}
while $M_\rho^2(X_1,\dots,X_n)=0$. For bivariate distance covariance, this bias has already been discussed by Cope \cite{Cope2009} and Sz\'ekely and Rizzo \cite{SzekRizz2009a}.
\end{remark}

Finally, we present a weak consistency result for $\MN_\rho$ under independence, which holds under milder moment conditions than the strong consistency result Theorem~\ref{thm:sconsistent}.
\begin{corollary}\label{cor:wconsistent}
    Suppose that $X_1,\dots,X_n$ are independent random variables with $\Ee\psi_i(X_i)<\infty$ and \\$\Ee\left[\log^{1+\epsilon}(1+|X_i|^2)\right]<\infty$ for some $\epsilon>0$ and all $i=1,\dots, n$. Then
    \begin{equation}
        \MN_\rho(\bm{X}^{(1)},\dots,\bm{X}^{(N)}) \xrightarrow[N\to \infty]{} 0 \quad\text{in probability.}
    \end{equation}
\end{corollary}
\begin{proof}
    The corollary is a direct consequence of Theorem \ref{thm:Mdistconv} and the observation that
    \begin{equation*}
        n Z_n \xrightarrow{d} Z
        \implies Z_n \xrightarrow{d} 0
        \implies Z_n \xrightarrow{\Pp} 0;
    \end{equation*}
    the second implication follows since the $d$-limit is degenerated.
\end{proof}

\subsection{Estimating total distance multivariance}\label{sec:estimate_total}
To simplify notation we write $\rho_S = \bigotimes_{i\in S}\rho_i$. Recall that
\begin{equation}\label{eq:totalms}
    \MNo_\rho^2(\bm{X}^{(1)},\dots,\bm{X}^{(N)})
    = \sum_{\substack{S\subset \{1,\dots,n\}\\ |S|\geq 2}} \MN_{\rho_S}^2(\bm{X}^{(1)},\dots,\bm{X}^{(N)}).
\end{equation}
Note that $M_{\rho_S}$ depends only on the random variables $(X_i, i\in S)$, i.e.\ $M_{\rho_S}= M_{\rho_S}(X_i, i\in S)$. This means that the sample version $\MN_{\rho_S} = \MN_{\rho_S}(\bm{X}^{(1)},\dots,\bm{X}^{(N)})$ is computed only from the $S$-coordinates of the samples $\bm{X}^{(1)},\dots,\bm{X}^{(N)}$.   The results of this section are mostly direct consequences of the results of the previous section (replacing $M_\rho$ by $M_{\rho_S}$ and $\MN_\rho$ by $\MN_{\rho_S}$).

\begin{corollary}[$\MNo_\rho$ is a strongly consistent estimator of $\overline M_\rho$]\label{cor:tMsconsistent}
    Assume that one of the moment conditions of Definition~\ref{def:moment} is satisfied. Then
    \begin{equation}
        \MNo_\rho(\bm{X}^{(1)},\dots,\bm{X}^{(N)}) \xrightarrow[N\to \infty]{} \overline{M}_\rho(X_1,\dots,X_n) \text{ a.s.}
    \end{equation}
\end{corollary}
\begin{proof}
    Apply Theorem \ref{thm:sconsistent} to each $M_{\rho_S}$ in \eqref{eq:totalms}.
\end{proof}

\begin{corollary}\label{cor:tMwconsistent}
    Let $X_1,\dots,X_n$ be independent random variables with $\Ee\psi_i(X_i)<\infty$ and \\$\Ee\left[\log^{1+\epsilon}(1+|X_i|^2) \right]<\infty$ for some $\epsilon>0$ and all $i=1,\dots, n$. Then
    \begin{equation}
        \MNo_\rho(\bm{X}^{(1)},\dots,\bm{X}^{(N)}) \xrightarrow[N\to \infty]{} 0 \quad\text{in probability}.
    \end{equation}
\end{corollary}
\begin{proof}
    Apply Corollary \ref{cor:wconsistent} to each $M_{\rho_S}$ in \eqref{eq:totalms}.
\end{proof}

The next theorem is the analogue of the convergence result Theorem~\ref{thm:Mdistconv}. For each $S\subset\{1,\dots,n\}$, we denote by $\mathds{G}_S$ the centred Gaussian process
\begin{equation}
    \mathds{G}_S(t_S) := \sum_{R\subset S} (-1)^{|S|-|R|} \int\ee^{\ii \scalp{x_R}{t_R}}\,\ddB(x) \cdot \prod_{\smash{j} \in S\backslash R} f_j(t_j),
\end{equation}
cf.~\eqref{eq:conv_to_G} in the supplement \cite{part2supp}, indexed by $t_S \in \bigtimes_{i\in S} \real^{d_i}$, and where $B$ is the Brownian bridge from \eqref{eq:convBB} in the supplement \cite{part2supp}. Applying Theorem~\ref{thm:Mdistconv} with $\{1, \dots, n\}$ replaced by $S$, we see that $\mathds{G}_S$ has covariance structure
\begin{equation}\label{eq:GS_cov}
\Ee(\mathds{G}_S(t)\overline{\mathds{G}_S(t')}) = \prod_{i \in S} \left(f_{X_i}(t_i-t_i') - f_{X_i}(t_i)\overline{f_{X_i}(t_i')}\right).
\end{equation}

\begin{theorem}[Asymptotic distribution of $\MNo_\rho$]\label{thm:Mdconv}\mbox{}

    \smallskip\textup{a)}
    Suppose that $X_1,\dots,X_n$ are independent with $\Ee\psi_i(X_i)<\infty$ and $\Ee\left[\log^{1+\epsilon}(1+|X_i|^2)\right] <\infty$ for some $\epsilon>0$ and all $i=1,\dots,n$. Then
    \begin{equation}\label{eq:empMdist}
        N\cdot\MNo_\rho^2(\bm{X}^{(1)},\dots,\bm{X}^{(N)})
        \xrightarrow[N\to \infty]{d} \sum_{\substack{S\subset \{1,\dots,n\}\\ |S| \geq 2}} \|\mathds{G}_S\|_{L^2(\rho_S)}^2.
    \end{equation}

    \smallskip\textup{b)}
    Suppose that the random variables $X_1, \dots, X_n$ are not independent and that one of the moment conditions of Definition~\ref{def:moment} holds. Then
    \begin{equation}
        N\cdot\MNo_\rho^2(\bm{X}^{(1)},\dots,\bm{X}^{(N)}) \xrightarrow[N\to\infty]{} \infty\quad\text{a.s.}
    \end{equation}
\end{theorem}
\begin{remark}\label{rem:gaussian_qf}
Note that the processes $(\mathds{G}_S), S \subset \{1, \dots, n\}$ on the right hand side of \eqref{eq:empMdist} are \emph{jointly Gaussian}. Therefore, the limit appearing in \eqref{eq:empMdist} is a quadratic form of centred Gaussian random variables. This fact will be used in Subsection~\ref{sec:test} to construct a statistical test of (multivariate) independence. Further properties of the processes $\mathds{G}_S$ are discussed in \cite{BersBoet2018}.
\end{remark}
\begin{proof}[Proof of Theorem~\ref{thm:Mdconv}]
a) \ For any $S\subset \{1,\dots,n\}$ with $|S|\geq 2$, we know from Theorem~\ref{thm:Mdistconv} that
\begin{equation*}
    N\cdot\MN_{\rho_S}^2(\bm{X}^{(1)},\dots, \bm{X}^{(N)})
    \xrightarrow[N\to\infty]{} \|\mathds{G}_S\|_{\rho_S}^2,
\end{equation*}
and \eqref{eq:empMdist} follows.

\smallskip
b) \ By Corollary \ref{cor:tMsconsistent} we have $\MNo_\rho \to \overline{M}_\rho$ almost surely. Moreover, $\overline{M}_\rho > 0$ by Theorem \ref{thm:independence}, since the random variables $(X_1, \dots, X_n)$ are not independent. Thus, $N\cdot\MNo_\rho^2 \to \infty$ almost surely.
\end{proof}

\subsection{Normalizing and scaling distance multivariance}\label{sub:scaling}

With practical applications in mind, there are at least two reasons to consider rescaled versions of (total) distance multivariance:

\begin{itemize}
\item To obtain a \emph{distance multicorrelation} whose %absolute
value is bounded by $1$ -- analogous to Sz\'ekely-Rizzo-and-Bakirov's distance correlation \cite[Def.~3]{SzekRizzBaki2007};
\item To normalize the asymptotic distribution of the sample (total) distance multivariance under independence, cf. Theorem~\ref{thm:Mdistconv} and Theorem~\ref{thm:Mdconv}.
\end{itemize}
We will use normalized multivariances as test statistics in two tests for independence in Section~\ref{sec:test}. For the scaling constants we use in the following the convention $0/0 := 0$. This ensures that we also cover the case of degenerated (i.e.~constant) random variables.

\subsubsection*{Distance multicorrelation}

\begin{definition}
    Let $X_1, \dots, X_n$ be random variables with
    $\Ee \psi_i^n(X_i) < \infty$ for all $i=1,\dots, n$. We set
   \begin{align*}
        &a_i
        := \norm{\Ce_{X_i} \Ce_{X_i'}\psi_i(X_i - X_i')}_{L^n(\Pp)}
    \end{align*}
    and define \emph{distance multicorrelation} as
    \begin{equation}
        \Rskript^2_\rho(X_1, \dots, X_n) := \frac{M^2_\rho(X_1, \dots, X_n)}{a_1 \cdot \ldots \cdot a_n}.
    \end{equation}
\end{definition}
For the sample version of distance multicorrelation, we define
\begin{equation}
    \hN a_i
    := \hN a_i( \bm{x}^{(1)}, \dots, \bm{x}^{(N)})
    = \bigg(\frac{1}{N^2} \sum_{k,l=1}^N |(A_i)_{kl}|^n \bigg)^{1/n},
\end{equation}
where the $A_i$ are the doubly centred matrices from Theorem~\ref{thm:sample}, and set
\begin{equation}
    \RN^2_\rho(\bm{x}^{(1)}, \dots, \bm{x}^{(N)})
    := \frac{1}{N^2} \sum_{k,l=1}^N \frac{(A_1)_{kl}}{\hN a_1} \cdot \ldots \cdot \frac{(A_n)_{kl}}{\hN a_n}.
\end{equation}

Note that $a_i = 0$ if, and only if, $X_i$ is degenerate, hence, $\Rskript_\rho(X_1, \dots, X_n)$ is well-defined as a finite non-negative number.
\begin{proposition}\label{prop:dis-mult}
\textup{a)}
    Distance multicorrelation and its sample version satisfy
    \begin{equation}\label{eq:R_bound}
        0 \leq \Rskript_\rho(X_1, \dots, X_n) \leq 1
        \quad\text{and}\quad
        0 \leq \RN_\rho (\bm{x}^{(1)}, \dots, \bm{x}^{(N)}) \leq 1.
    \end{equation}

\smallskip\textup{b)}
    For iid copies $\bm{X}^{(1)}, \dots,  \bm{X}^{(N)}$ of $\bm{X} = (X_1, \dots, X_n)$ it holds that
    $$
        \lim_{N \to \infty} \RN_\rho (\bm{X}^{(1)}, \dots, \bm{X}^{(N)})
        = \Rskript_\rho (X_1, \dots, X_n),
        \quad\text{a.s.}
    $$

\smallskip\textup{c)}
    For $n=2$ and $\psi_1(x) = \psi_2(x) = |x|$ distance multicorrelation coincides with the distance correlation of \cite{SzekRizzBaki2007}.
\end{proposition}

\begin{remark}
    Sz\'ekely and Rizzo \cite[Thm~4.(iv)]{SzekRizz2009} show for the case $n=2$ (i.e.\ for distance correlation) that $\RN_\rho(\bm{X}^{(1)}, \dots, \bm{X}^{(N)}) = 1$ implies that the sample points $(x_1^{(1)}, \dots, x_1^{(N)})$ and $(x_2^{(1)}, \dots,x_2^{(N)})$ can be transformed into each other by a Euclidean isometry composed with scaling by a non-negative number. An analogous result seems not to hold for distance multicorrelation in the case $n > 2$.
\end{remark}
\begin{proof}[Proof of Proposition~\ref{prop:dis-mult}]
    By the generalized H\"older inequality for $n$-fold products (cf.~\cite[p.~133, Pr.~13.5]{schilling-mims}),
    we have that
    \begin{align*}
        M^2_\rho(X_1, \dots, X_n)
        &= \Ee\left(\prod_{i=1}^n - \Ce_{X_i} \Ce_{X'_i} \psi_i(X_i - X'_i)\right)\\
        &\leq \Ee\left(\prod_{i=1}^n \left|\Ce_{X_i} \Ce_{X'_i} \psi_i(X_i - X'_i)\right|\right)\\
        &\leq \prod_{i=1}^n \left\|\Ce_{X_i} \Ce_{X_i'}\psi_i(X_i - X_i')\right\|_{L^n(\Pp)} = a_1 \cdot \ldots \cdot a_n,
    \end{align*}
    and \eqref{eq:R_bound} follows.
    For the convergence result, note that
    \begin{equation}\label{eq:empEpsi}
        \hN a_i^n
        = \frac{1}{N^2} \sum_{k,l =1}^N \left|(A_i)_{kl} \right|^n \xrightarrow[N\to \infty]{}\Ee\left[\left|\Ce_{X_i} \Ce_{X_i'}\psi_i(X_i-X_i')\right|^n\right]
        = a_i^n
    \end{equation}
    by the law of large numbers for V-statistics, cf.~Theorem~\ref{thm:sconsistent} and its proof. Part c) follows from direct comparison with \cite{SzekRizz2009}.
\end{proof}

\subsubsection*{Normalized distance multivariance}

Alternatively, we can normalize distance multivariance in such a way, that the limiting distribution under independence (cf.~ Theorems~\ref{thm:Mdistconv} and \ref{thm:Mdconv}) has unit expectation.
\begin{definition}
    Let $X_1, \dots, X_n$ be random variables with $\Ee\psi_i(X_i) < \infty$ for all $i=1, \dots, n$, set
    \begin{equation*}
        b_i := \Ee \psi_i(X_i - X_i')
    \end{equation*}
    and define \emph{normalized distance multivariance} as
    \begin{equation}
        \Mskript^2_\rho(X_1, \dots, X_n) := \frac{M_\rho^2(X_1, \dots, X_n)}{b_1 \cdot \ldots \cdot b_n}.
    \end{equation}
\end{definition}
    For the sample version of normalized distance multicorrelation, we define
    \begin{equation}
        \hN b_i
        := \hN b_i( \bm{x}^{(1)}, \dots, \bm{x}^{(N)})
        := \frac{1}{N^2} \sum_{k,l=1}^N \psi_i\left(x_i^{(l)}  - x_i^{(k)}\right)
        = \frac{1}{N^2} \sum_{k,l=1}^N (B_i)_{kl},
    \end{equation}
and set
\begin{equation}
    \MsN_\rho^2 (\bm{x}^{(1)}, \dots, \bm{x}^{(N)})
    := \frac{1}{N^2} \sum_{k,l=1}^N \frac{(A_1)_{kl}}{\hN b_1} \cdot \ldots \cdot \frac{(A_n)_{kl}}{\hN b_n}.
\end{equation}

\begin{corollary}\label{cor:Q_conv}
    Suppose that $X_1,\dots,X_n$ are non-degenerate independent random variables with $\Ee\psi_i(X_i)<\infty$ and $\Ee\left[\log^{1+\epsilon}(1+|X_i|^2) \right]<\infty$ for some $\epsilon>0$ and all $i=1,\dots, n$. Then
    \begin{equation}\label{eq:S_conv}
        N \cdot \MsN_\rho^2(\bm{X}^{(1)},\dots,\bm{X}^{(N)}) \xrightarrow[N\to \infty]{d} Q,
    \end{equation}
    where $Q = \|\mathds{G}\|_\rho^2 / (b_1 \cdot \ldots \cdot b_n)$ and $\Ee Q = 1$.
\end{corollary}
\begin{proof}
    This follows from Theorem \ref{thm:Mdistconv} in combination with
    \begin{equation}\label{eq:empEpsi-2}
        \hN b_i = \frac{1}{N^2} \sum_{k,l =1}^N \psi_i(X_i^{(k)}-X_i^{(l)}) \xrightarrow[N\to \infty]{}\Ee\psi_i(X_i-X_i') = b_i,
    \end{equation}
    under the assumption $\Ee\psi_i(X_i)<\infty$.
\end{proof}

It remains to find an analogous normalization for \emph{total} distance multivariance. For a subset $S \subset \{1, \dots, n\}$ define $M_{\rho_S}(X_1, \dots, X_n)$ as in Section~\ref{sec:estimate_total} and set $b_S = \prod_{i \in S} b_i$.

\begin{definition}
For the random variables $X_1, \dots, X_n$ we define the \emph{normalized total distance multivariance} as
\begin{equation}\label{eq:norm_total_DM}
    \overline{\Mskript}^2_\rho(X_1, \dots, X_n)^2
    := \frac{1}{2^n - 1 - n}\sum_{\substack{S\subset \{1,\dots,n\}\\ |S| \geq 2}}\frac{ M_{\rho_S}^2(X_i, i\in S)}{b_S}.
\end{equation}
Its sample version becomes
\begin{small}
\begin{align}
    &\hN\overline{\Mskript}^2_{\rho} (\bm{x}^{(1)}, \dots, \bm{x}^{(N)})\\
    \notag &\quad:= \frac{1}{2^n - 1 - n} \left\{\frac{1}{N^2} \sum_{k,l=1}^N \left(1 + \frac{(A_1)_{kl}}{\hN b_1}\right) \cdot \ldots \cdot \left(1 + \frac{(A_n)_{kl}}{\hN b_1}\right) - 1\right\}.
\end{align}
\end{small}
\end{definition}

Similar to Corollary~\ref{cor:Q_conv}, we have the following result.
\begin{corollary}\label{cor:Q_conv_total}
    Suppose that $X_1,\dots,X_n$ are non-degenerate independent random variables with $\Ee\psi_i(X_i)<\infty$ and $\Ee\left[\log^{1+\epsilon}(1+|X_i|^2) \right]<\infty$ for some $\epsilon>0$ and all $i=1,\dots, n$. Then
    \begin{equation}\label{eq:S_conv_total}
        N \cdot \hN\overline{\Mskript}^2_\rho(\bm{X}^{(1)},\dots,\bm{X}^{(N)}) \xrightarrow[N\to \infty]{d} \overline{Q},
    \end{equation}
    where
    $$
        \overline{Q} = \frac{1}{2^n - n - 1} \sum_{\substack{S\subset \{1,\dots,n\}\\ |S| \geq 2}} \frac{\|\mathds{G}_S\|_{L^2(\rho_S)}^2}{b_S}
        \quad\text{and}\quad
        \Ee \overline{Q} = 1.
    $$
\end{corollary}
\begin{proof}
    Convergence follows from Theorem~\ref{thm:Mdconv}. Note that the sum runs over $2^n - n - 1$ subsets and $\Ee\left[\|\mathds{G}_S\|_{L^2(\rho_S)}^2\right] = b_S$ by Corollary~\ref{cor:Q_conv}.
\end{proof}

\subsection{Two tests for independence}\label{sec:test}
Based on the normalized multivariance statistics $\Mskript_\rho$ and $\overline{\Mskript}_\rho$ and the convergence results of Corollaries~\ref{cor:Q_conv} and \ref{cor:Q_conv_total}, we can formulate two statistical tests for the independence of the random variables $X_1,\dots, X_n$. To assess a critical value for the test statistics, we use the same approach as Sz\'ekely and Rizzo \cite{SzekRizz2009}: Both limiting random variables $Q$ and $\overline{Q}$ are quadratic forms of centred Gaussian random variables, normalized to $\Ee Q = \Ee\overline{Q} = 1$. Hence, by \cite[p.~181]{SzekBaki2003},
\begin{equation} \label{eq:Qalpha}
    \Pp(Q \ge \chi_{1-\alpha}^2(1)) \le \alpha
    \quad\text{and}\quad
    \Pp(\overline{Q} \ge \chi_{1-\alpha}^2(1)) \le \alpha,
\end{equation}
for all $0 < \alpha \le 0.215$, where $\chi_{1-\alpha}^2(1)$ denotes the $(1- \alpha)$-quantile of a chi-square distribution with one degree of freedom. Note that \eqref{eq:Qalpha} is, in general, very rough, thus the following tests are, in general, quite conservative.  The first test uses multivariance and, therefore, requires the a-priori assumption of $(n-1)$-independence.

\begin{test} \label{testA}
    Let $\bm{x}^{(1)}, \dots, \bm{x}^{(N)}$ be observations of the random vector $\bm{X} = (X_1, \dots, X_n)$, let $\alpha \in (0,0.215)$, and suppose that  the moment conditions of Corollary~\ref{cor:Q_conv} and one of the moment conditions of Definition~\ref{def:moment} hold.  Under the assumption of $(n-1)$-independence, the null hypothesis
    \begin{align*}
        \bm{H_0}:&\qquad (X_1, \dots, X_n) \quad\text{are independent}
    \intertext{is rejected against the alternative hypothesis}
        \bm{H_1}:&\qquad (X_1, \dots, X_n) \quad\text{are not independent}
    \end{align*}
    at level $\alpha$, if the normalized multivariance $\MsN(\bm{x}^{(1)}, \dots, \bm{x}^{(N)})$ satisfies
    $$
        N \cdot \MsN^2_\rho(\bm{x}^{(1)}, \dots, \bm{x}^{(N)})
        \geq \chi^2_{1 - \alpha}(1).
    $$
\end{test}
The second test uses \emph{total} multivariance, and hence does not require a-priori assumptions, except for the moment conditions. We emphasize that this test on mutual independence can be applied in very general settings: It is distribution-free and the random variables $X_1, \dots, X_n$ can take values in arbitrary dimensions.

\begin{test} \label{testB}
    Let $(\bm{x}^{(1)}, \dots, \bm{x}^{(N)})$ be observations of the random vector $\bm{X} = (X_1, \dots, X_n)$, let $\alpha \in (0,0.215)$, and suppose that  the moment conditions of Corollary~\ref{cor:Q_conv_total} and one of the moment conditions of Definition~\ref{def:moment} hold.  The null hypothesis
    \begin{align*}
        \bm{H_0}:&\qquad (X_1, \dots, X_n) \quad\text{are independent}
    \intertext{is rejected against the alternative hypothesis}
        \bm{H_1}:&\qquad (X_1, \dots, X_n) \quad\text{are not independent}
    \end{align*}
    at level $\alpha$, if the normalized total multivariance $\hN\overline{\Mskript}_\rho(\bm{x}^{(1)}, \dots, \bm{x}^{(N)})$ satisfies
    $$
        N \cdot \hN\overline{\Mskript}^2_\rho(\bm{x}^{(1)}, \dots, \bm{x}^{(N)}) \ge \chi^2_{1 - \alpha}(1).
    $$
\end{test}
Note that in Test \ref{testA} and Test \ref{testB} the moment conditions of Definition~\ref{def:moment} ensure the divergence (for $N \to \infty$) of the test statistics in the case of dependence, cf.~Theorem \ref{thm:Mdistconv} and Theorem \ref{thm:Mdconv}. Thus these tests are consistent against all alternatives.

In Section~\ref{ex:bernstein} below we give a numerical example of both tests that also allows to assess their power for different sample sizes $N$.

\begin{remark}
    If the marginal distributions are known, it is possible to perform a Monte Carlo test, where the $p$-value is obtained from the empirical (Monte Carlo) distribution of the test statistic under $H_0$. Even without knowledge of the marginal distributions, resampling tests can be performed. These and further tests based on distance multivariance are discussed in \cite{Boet2017,BersBoet2018}.
\end{remark}

\section{Examples}\label{ex:bernstein}

In this section we present two basic examples which illustrate some key aspects of distance multivariance:

\begin{itemize}
\item \textbf{Bernstein's coins:} This is a classical example of pairwise independence with higher order dependence. It shows that distance multivariance accurately detects multivariate dependence.
\item \textbf{Sinusoidal dependence:} This is a basic example which was considered in \cite{SejdSripGretFuku2013} to illustrate that distance covariance can perform poorly when used to detect small scale (local) dependencies. We show that the flexibility of generalized distance multivariance  -- due to the choice of the distance functions $\psi_i$ -- can be used to improve the power of the test considerably.
\end{itemize}

\subsection{Bernstein's coins}
The first example of pairwise independent, but not (totally) independent random variables is attributed to S.N.~Bernstein, cf.~\cite[Sec.~V.3]{feller1971introduction}. We illustrate this example by using two identical fair coins, coin I and coin II. Based on independent tosses of these two coins, define the following events
\begin{gather*}
    A = \{\text{coin I shows heads}\}, \qquad
    B = \{\text{coin II shows tails}\}, \\
    C = \{\text{both coins show the same side}\}.
\end{gather*}
All events have probability $\frac{1}{2}$, and they are pairwise independent, since
$$
    \Pp(A \cap B) = \Pp(B \cap C) = \Pp(C \cap A) = \frac{1}{4}.
$$
They are, however, not independent, since $A \cap B \cap C = \emptyset$, hence
$$
    0 = \Pp(A \cap B \cap C) \neq \Pp(A)\cdot\Pp(B)\cdot\Pp(C) = \frac{1}{8}.
$$

Hence, the distance covariances\footnote{In slight abuse of notation, we identify the events $A,B,C$ with the random variables $\One_A(\omega), \One_B(\omega), \One_C(\omega)$.} of the pairs $(A,B)$, $(B,C)$ and $(C,A)$ should vanish, due to pairwise independence, while the distance multivariance and the total distance multivariance of the triplet $(A,B,C)$ should detect their higher-order dependence. We discuss both the analytic approach and the numerical simulation of the relevant quantities.

Let $\rho_A, \rho_B, \rho_C$ be one-dimensional symmetric L\'evy measures with the corresponding continuous negative definite functions $\psi_A$, $\psi_B$ and $\psi_C$. We write  $\rho = \rho_A\otimes\rho_B\otimes\rho_C$ and $\rho_{AB}:=\rho_A\otimes\rho_B$ etc.

\subsubsection*{Analytic Approach}
First, note that pairwise independence yields
\begin{align*}
    M_{\rho_{AB}}(A,B)^2
    &= \int_{\real^2} \left(f_{A,B}(r,s)  - f_A(r)f_B(s)\right)^2 \rho_A\otimes\rho_B(\ddr,\dds) \\
    &= \int_{\real}\int_{\real} 0\,\rho_A(\ddr)\,\rho_B(\dds) = 0,
\end{align*}
and similarly for $M_{\rho_{BC}}(B,C)$ and $M_{\rho_{AC}}(C,A)$. On the other hand, from the pairwise independence and Corollary~\ref{cor:mInd-charfun} we obtain
\begin{align*}
    M_\rho&(A,B,C)^2
    = \int_{\real^3} \left(f_{A,B,C}(r,s,t)  - f_A(r)f_B(s)f_C(t)\right)^2 \rho(\ddr,\dds,\ddt)\\
    &= \frac{1}{64} \int_\real \int_\real \int_\real |1 -\ee^{\ii r} |^2 |1 -\ee^{\ii s} |^2 |1 -\ee^{\ii t} |^2 \,\rho_A(\ddr) \,\rho_B(\dds) \,\rho_C(\ddt)\\
    &= \frac{1}{8} \psi_A(1) \psi_B(1) \psi_C(1).
\end{align*}

\begin{figure}[htbp]
	\centering
%	    \subfigure[Multivariance without normalization]{\includegraphics[width=0.8\textwidth]{bernstein1-800x422.png}}
%        \subfigure[Normalized multivariance]{\includegraphics[width=0.8\textwidth]{bernstein2-800x422.png}}
%        \subfigure[Squared normalized multivariance scaled by sample size]{\includegraphics[width=0.8\textwidth]{bernstein3-800x422.png}}
	    \subfigure[Multivariance without normalization]{\includegraphics[width=0.8\textwidth]{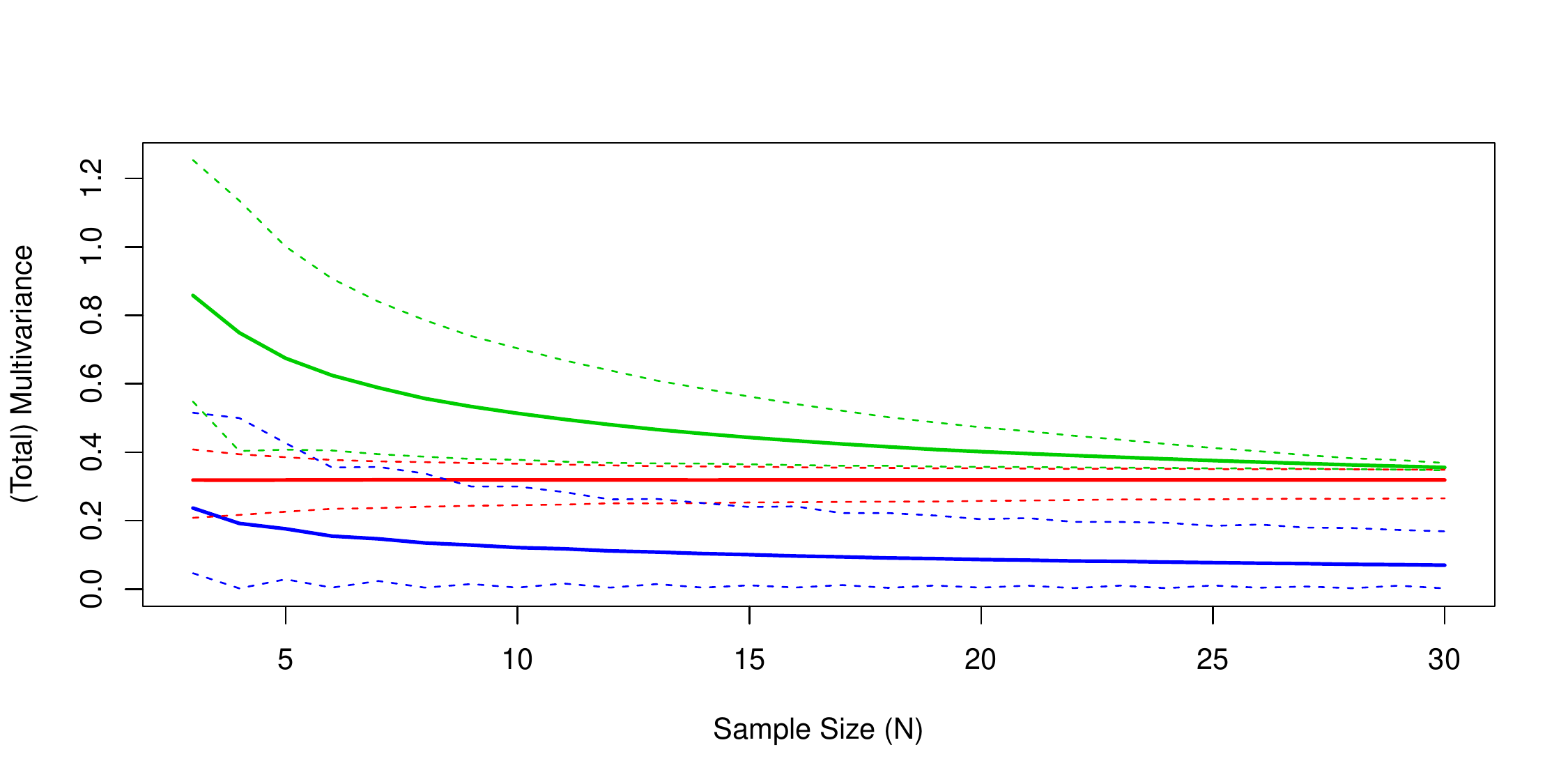}}
        \subfigure[Normalized multivariance]{\includegraphics[width=0.8\textwidth]{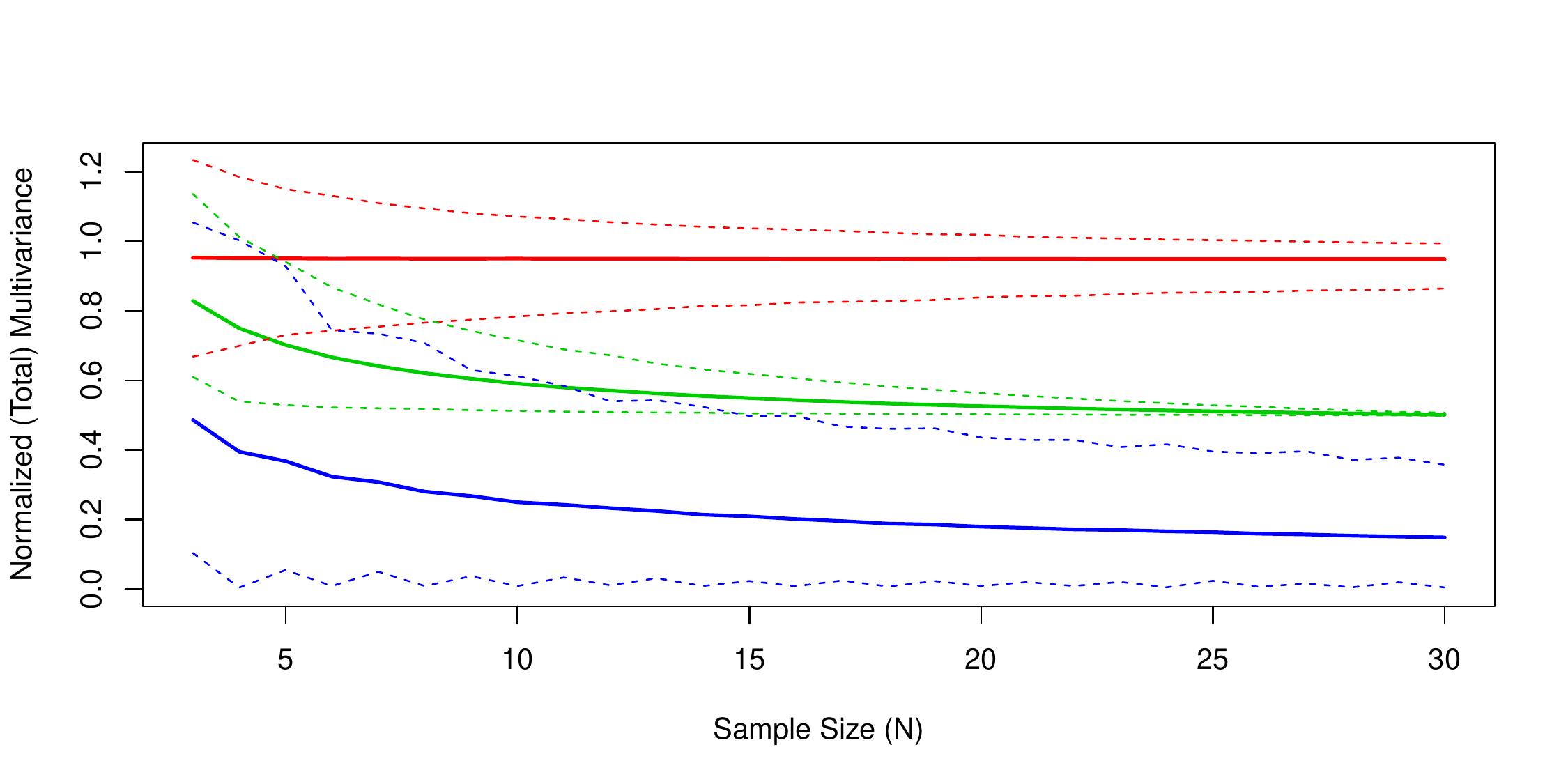}}
        \subfigure[Squared normalized multivariance scaled by sample size]{\includegraphics[width=0.8\textwidth]{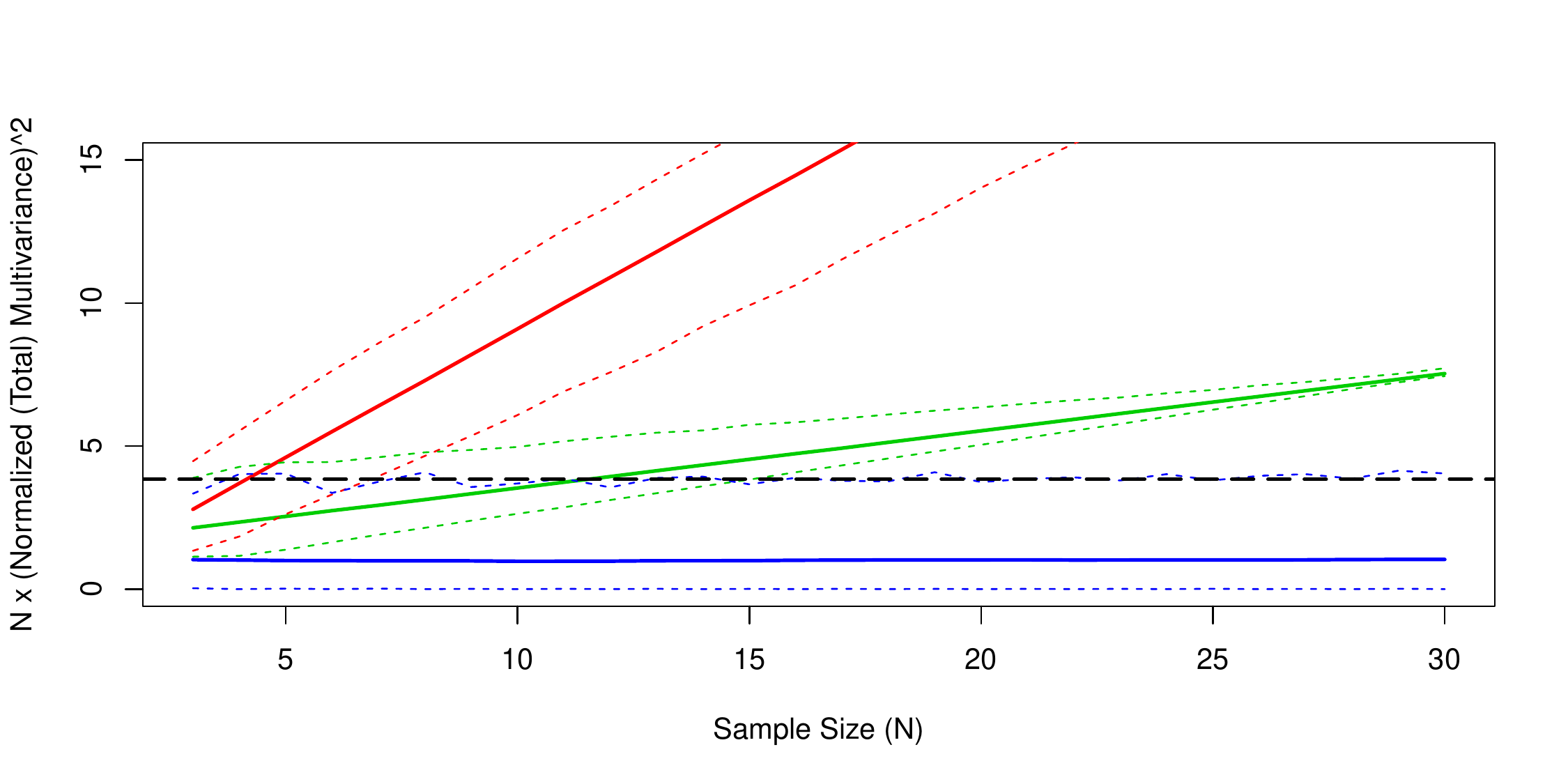}}
    \caption{These plots show sample distance covariance $\MN_{\rho_{AB}}(A,B)$ \textup{(}blue\textup{)}, sample distance multivariance $\MN_\rho(A,B,C)$ \textup{(}red\textup{)} and sample total distance multivariance $\MNo_\rho(A,B,C)$ \textup{(}green\textup{)} for Bernstein's coin toss experiment, cf.~Section~\ref{ex:bernstein}, averaged over $5000$ Monte-Carlo replications. Also shown are the empirical $5\%$ and $95\%$ quantiles \textup{(}dashed\textup{)}. Different scalings are used in the plots \textup{(a)--(c)}, and plot \textup{(c)} also shows the critical value \textup{(}significance level $\alpha=5\%$\textup{)} of the independence tests from Section~\ref{sec:test} \textup{(}long dashes, black\textup{)}.}
	\label{fig:bernstein}
\end{figure}

In particular, for $\psi(x) = |x|$ we obtain
$$
    M_\rho(A,B,C) = \overline{M}_\rho(A,B,C) = \frac{1}{2\sqrt{2}}.
$$
We calculate the scaling factors from Section~\ref{sub:scaling} as
$$
    a_A = a_B = a_C = b_A = b_B = b_C = \frac{1}{2},
$$
which shows that multicorrelation and normalized multivariance coincide in this case, i.e.
$$
    \Rskript_\rho(A,B,C)  = 1 = \Mskript_\rho(A,B,C).
$$
Finally, normalized total multivariance is given by
$$
    \overline{\Mskript}_\rho(A,B,C)  =  \frac{1}{\sqrt{2^3 - 3 - 1}} \,\Mskript_\rho(A,B,C) = \frac{1}{2}.
$$

\subsubsection*{Numerical Simulation}
To complement the analytical results by a numerical simulation, we have simulated $5000$ replications of $N=3,\dots,30$ tosses of Bernstein's coins. We calculated the pairwise sample distance covariances $\MN_{\rho_{AB}}(A,B)$, $\MN_{\rho_{BC}}(B,C)$, $\MN_{\rho_{AC}}(C,A)$ as well as the sample distance multivariance $\MN_\rho(A,B,C)$ and the sample total distance multivariance $\MNo_\rho(A,B,C)$. We used Euclidean distance as underlying distance in all cases. Due to pairwise independence, the bivariate distance covariances should tend to zero for increasing $N$, while the multivariances should tend to the non-zero limits that we calculated analytically above.

Figure~\ref{fig:bernstein} shows the average values of the multivariance statistics over $5000$ replications, along with their empirical $5\%$ and $95\%$ quantiles. Figure (a) uses no scaling, Figure (b) shows `normalized' quantities (cf.~Section~\ref{sub:scaling}) and Figure (c) shows squared normalized quantities scaled by $N$, as they appear in  Theorems~\ref{thm:Mdistconv} and~\ref{thm:Mdconv}. Also shown is the critical value $\chi^2_{0.95}(1)$ of the test proposed in Section~\ref{sec:test}. In summary, the numerical simulation shows that\\
$\bullet$   (Total) distance multivariance is able to distinguish correctly pairwise independence of the events $A,B,C$ from their higher-order dependence;\\
$\bullet$    The sample statistics converge quickly to their analytic limits and numerically confirm the asymptotic results from Theorems~\ref{thm:Mdistconv} and~\ref{thm:Mdconv}.\\
$\bullet$   The hypothesis of pairwise independence of $A$ and $B$ would be correctly accepted in about $95\%$ of simulations, confirming the specificity of the proposed tests.\\
$\bullet$   Test A (with the a-priori assumption of pairwise independence) has a power exceeding $95\%$ for sample sizes $N>5$. Test B (no a priori assumptions) has a power exceeding $95\%$ for $N > 14$.

Note that all necessary functions and tests for such simulations and for the use of distance multivariance in applications are provided in the R package \texttt{multivariance} \cite{Boett2017R-1.0.5}.

\subsection{Sinusoidal dependence}
In \cite[p. 2287]{SejdSripGretFuku2013} it was pointed out that for random variables $X$, $Y$ with a common \emph{sinusoidal density}
\begin{equation} \label{eq:sinusoidal}
    f_l(x,y) := \frac{1}{4 \pi^2}(1+ \sin(lx)\sin(ly)) \text{ on } [-\pi,\pi]^2 \quad\text{for some $l\in \nat$}
\end{equation}
the detection of the dependence using distance covariance is poor for $l>1$. It was also noted that choosing (in our notation) $\psi_i(x) = |x|^\alpha$ with some $\alpha \neq 1$ might improve the power, see Figure~\ref{fig:sinusoidal}.(a).  Using the bounded continuous negative definite function $\psi_i(x) = \frac{1}{\gamma}  (1 - \exp(-\gamma|x|))$ with $\gamma > 0$ can increase the power considerably for larger $l$, see Figure \ref{fig:sinusoidal}.(b). Here we used the same sample parameters as in \cite{BerrSamw2017} (5000 samples, $N = 200$, $\alpha = 0.05$). The $p$-values were calculated by Monte Carlo estimation with 10000 replications.

The following heuristic was used to choose the value of $\gamma$: Note that
\begin{equation}
    \psi_i(x) := \frac{1}{\gamma}  (1 - \exp(-\gamma|x|))
\end{equation}
is a bounded function which is strictly increasing for $x> 0$. Suppose we know that the local dependencies occur in a window of (Euclidean) distance $\delta$. Thus, it seems reasonable to \emph{neglect} all pairs which are further apart than $\delta$ by setting all their $\psi_i$-distances to (roughly) the same value, i.e.~we choose $\gamma$ such that $\psi_i(\delta) \geq 0.99 \cdot \sup_x \psi_i(x).$ This is achieved by setting $\gamma := - \ln(0.01)/\delta$. For the sinusoidal example $\delta$ is the period of the sin functions, i.e. $\delta=\pi/l$. Let us compare the resulting test with the methods \texttt{MINT} and \texttt{MINTav} which were proposed in [BS17] for a wide range of situations. Figure \ref{fig:multiMINT} shows in the setting of sinusoidal data that our proposed test outperforms \texttt{MINTav} and has similar power as the \textit{oracle test} \texttt{MINT}.
Note that \texttt{MINTav} uses no a priori information about the dependence scale, and that \texttt{MINT} computes the p-value using all possible parameters and selects a posteriori the parameter (for each setting) which yielded the highest power. In contrast, our test requires a heuristic parameter selection using certain a priori knowledge of the data generation mechanism. 

Further extensions and details on resampling, Monte Carlo and other tests based on distance multivariance can be found in \cite{Boet2017, BersBoet2018}.

\begin{figure}[htbp]
	\centering
%	    \subfigure[$\psi_i(x)=|x|^\alpha$]{\includegraphics[width=0.45\textwidth]{sinusoidal-alphastable.png}}
%        \subfigure[$\psi_i(x) = \frac{1}{\gamma}  (1 - \exp(-\gamma|x|))$]{\includegraphics[width=0.45\textwidth]{sinusoidal-bounded.png}}
	    \subfigure[$\psi_i(x)=|x|^\alpha$]{\includegraphics[width=0.45\textwidth]{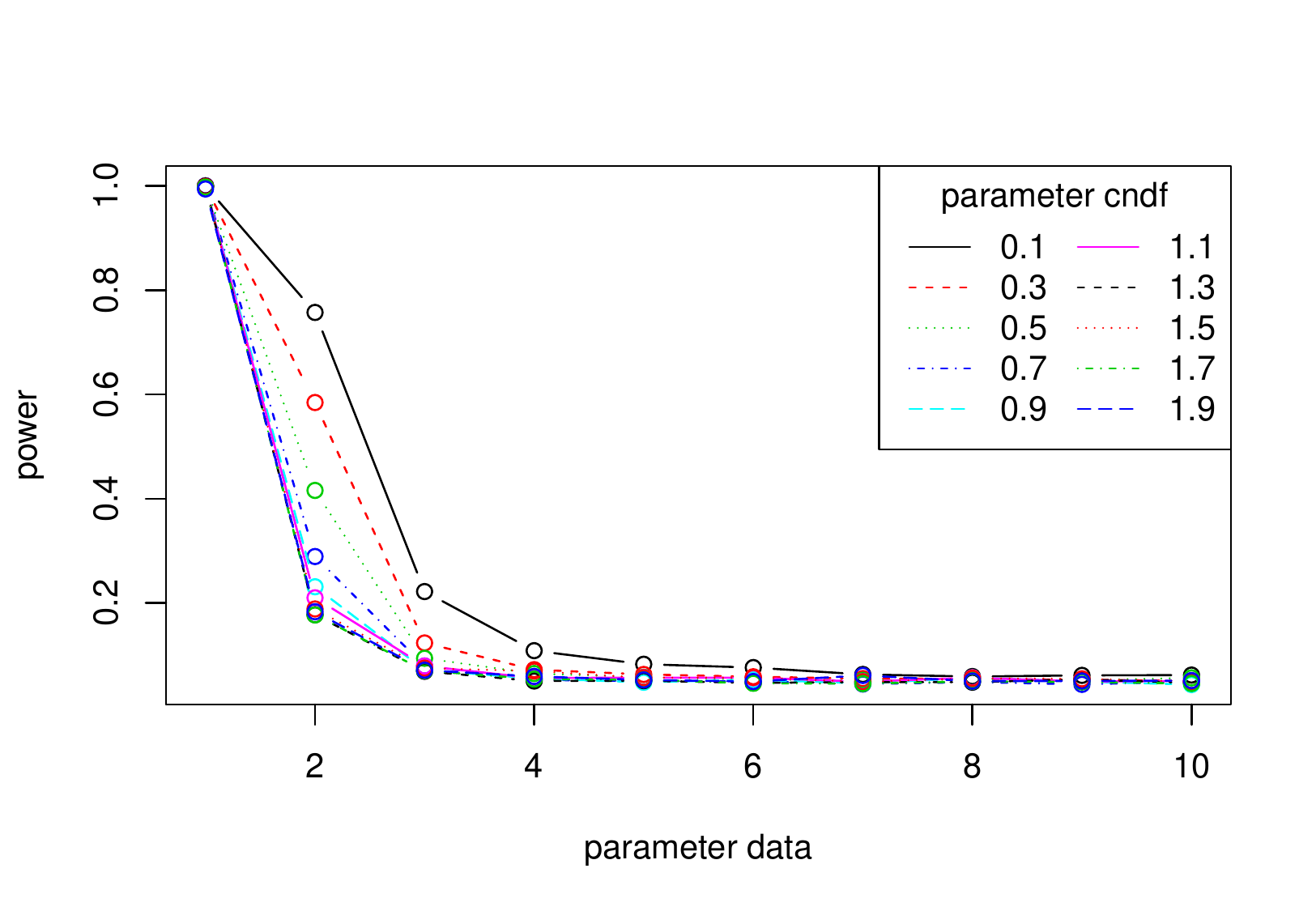}}
        \subfigure[$\psi_i(x) = \frac{1}{\gamma}  (1 - \exp(-\gamma|x|))$]{\includegraphics[width=0.45\textwidth]{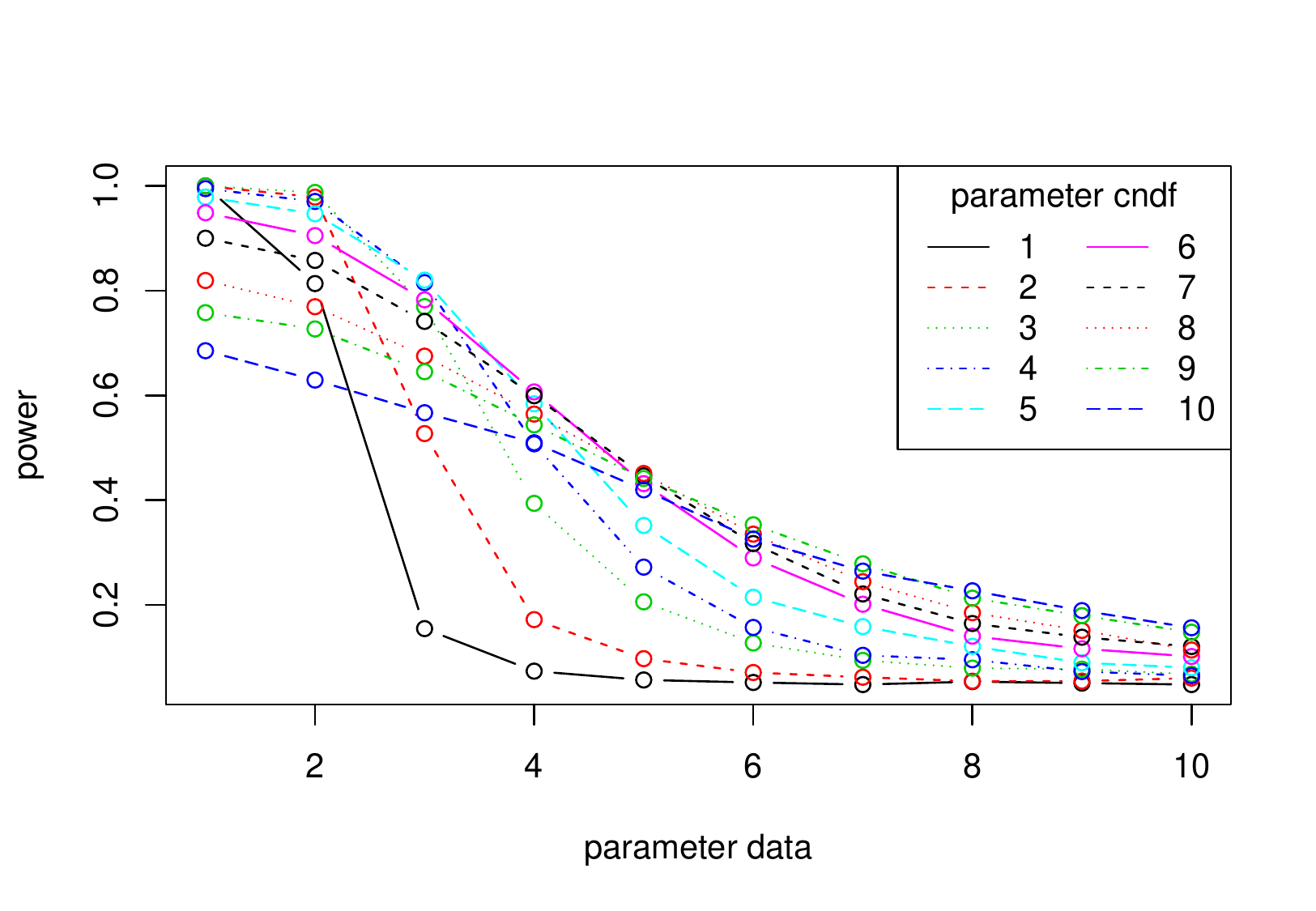}}
    \caption{Power of tests based on distance multivariance for the sinusoidal example with density $f_l$ given in \eqref{eq:sinusoidal}. The parameter of the data is $l$ and the parameter of the $\psi_i$ is $\alpha$ and $\gamma$, respectively. Here (a) is the alpha-stable case and (b) uses a bounded cndf.}
	\label{fig:sinusoidal}
\end{figure}

\begin{figure}[htbp]
\centering
\includegraphics[width=0.8\textwidth]{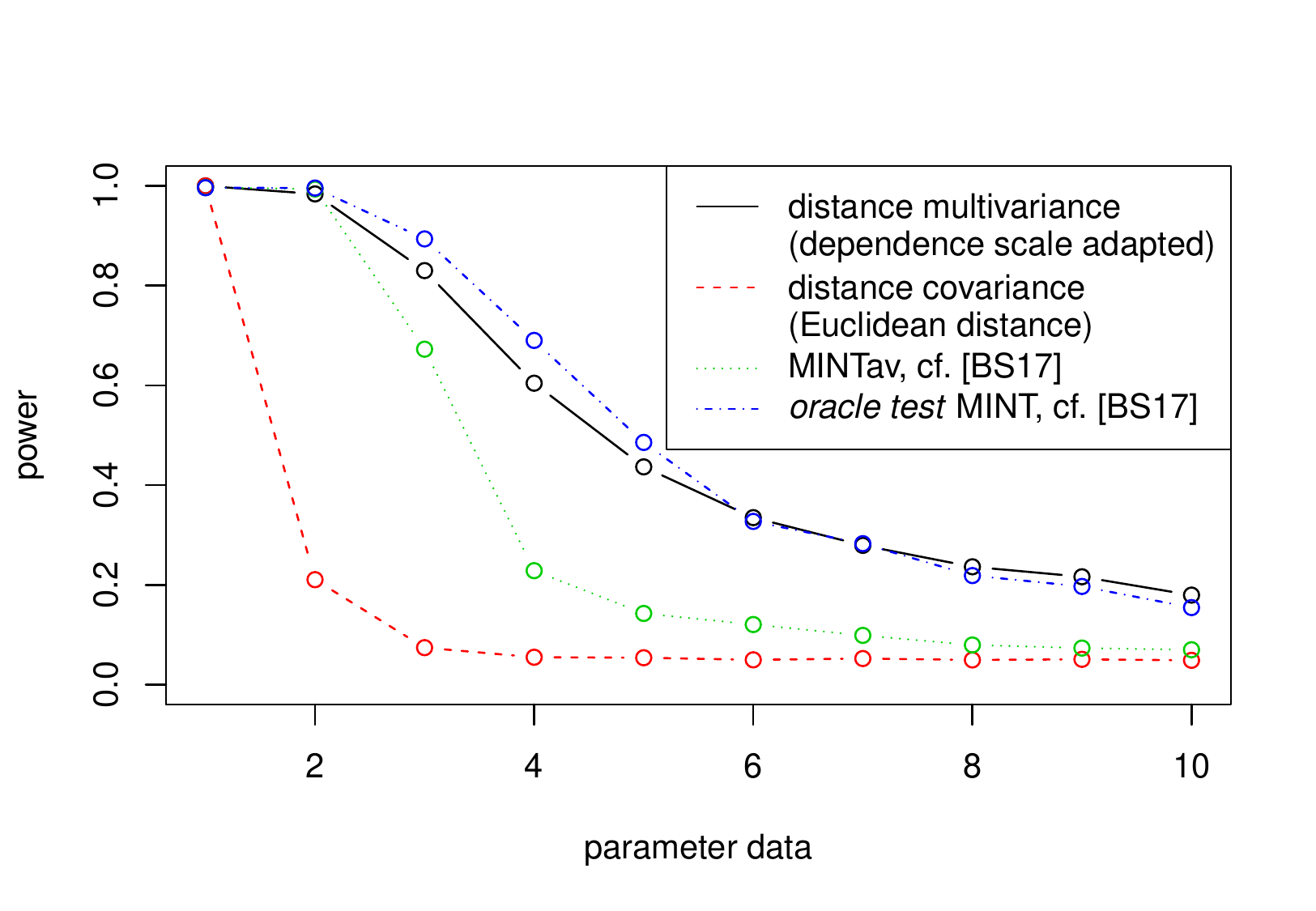}
%\caption{Comparison of the power of distance multivariance with distance adapted to the dependence scale, classical distance covariance with Euclidean distance and \textup{\texttt{MINTav}}. The latter was recently introduced in \textup{\cite{BerrSamw2017}} and it was shown that it outperforms many (all in their comparison) other dependence measures.}
\caption{Comparison of the power of distance multivariance with distance adapted to the dependence scale, classical distance covariance with Euclidean distance, \textup{\texttt{MINTav}} and \textup{\texttt{MINT}}. The latter were recently introduced in \textup{\cite{BerrSamw2017}} and it was shown that for this example they outperform many (all in their comparison) other dependence measures.}
\label{fig:multiMINT}
\end{figure}

\section*{Acknowledgments}

We are grateful to Ulrich Brehm (TU Dresden), for insightful discussions on (elementary) symmetric polynomials and to Georg Berschneider (TU Dresden) who read and commented on the whole text. We would also like to thank Gabor J.\ Sz\'ekely (NSF) for advice on the current literature. We thank the anonymous referees and the handling editor for their helpful comments.

Martin Keller-Ressel acknowledges support by the German Research Foundation (DFG) under grant ZUK~64 and KE~1736/1-1.

\setcounter{part}{18}

\part{Supplement to ``Distance multivariance: New dependence measures for random vectors''}
%\addcontentsline{toc}{part}{\protect\numberline{}Supplement to ``Distance multivariance: New dependence measures for random vectors''}%

%\setcounter{section}
\renewcommand{\thesection}{\Alph{part}}
\section{Proofs and auxiliary results}

Here we collect supplementary material to \cite{part2}. It contains the proofs of some of the main results as well as a few additional statements: Lemma \ref{lem:moments} discusses the moment conditions introduced in Definition \ref{def:moment} and Lemma \ref{lem:ZN} analyses the estimator which is required for the proof of the main convergence result (Theorem~\ref{thm:Mdistconv}).

Unless otherwise mentioned, all numbered references refer to \cite{part2}.

\subsection{Proofs and auxiliary results for Section~\ref{sec:prelim}}

\begin{proof}[Proof of Lemma~\ref{lem:factorization}]
    For arbitrary $a_i,b_i\in\comp$, $i=1,\dots, n$, we have
    \begin{equation}\label{eq:ab_identity}
        \prod_{i=1}^n(a_i-b_i) = \sum_{S\subset \{1,\dots,n\}} \left(\prod_{i \in S} a_i\right) \left(\prod_{i \in S^c} b_i\right) (-1)^{|S^c|},
    \end{equation}
    where $|S|$ denotes the cardinality of $S$ and $S^c:= \{1,\dots,n\}\setminus S$. Thus,
    \begin{align*}
        &\Ee\left(\prod_{i=1}^n (Z_i - \Ee Z_i)\right)
        = \Ee\left[\sum_{S\subset \{1,\dots,n\}} \left(\prod_{i \in S} Z_i\right) \left(\prod_{i \in S^c} \Ee(Z_i)\right) (-1)^{|S^c|}\right]\\
        &= \Ee\left(\prod_{i=1}^n Z_i\right) + \sum_{\substack{S\subset \{1,\dots,n\}\\|S|\leq n-1}}\Ee\left( \prod_{i \in S}  Z_i\right)  \left(\prod_{i \in S^c} \Ee\left( Z_i\right)\right) (-1)^{|S^c|}\\
        &= \Ee\left(\prod_{i=1}^n Z_i\right) + \sum_{\substack{S\subset \{1,\dots,n\}\\|S|\leq n-1}} \left(\prod_{i=1}^n \Ee\left( Z_i\right)\right) (-1)^{|S^c|}\\
        &= \Ee\left(\prod_{i=1}^n Z_i\right) - \prod_{i =1}^n\Ee( Z_i);
    \end{align*}
    $(n-1)$-independence is used in the penultimate line.
\end{proof}

\begin{lemma}\label{lem:moments}
    The moment conditions in Definition~\ref{def:moment} are ordered from weak to strong, i.e.~c) implies b) and b) implies a). In particular, the estimate
    \begin{equation}\label{eq:prod_est}
  \Ee\left( \prod_{i=1}^n \psi_i(X_{k_i,i}-X'_{l_i,i})\right) \leq 4^n \prod_{i=1}^n \left(\Ee \psi_i^{p_i}(X_{i})\right)^{1/p_i}
\end{equation}
holds for all $k_i, l_i \in \{0,1\}, i=1, \dots, n$ and all $p_i \in [1,\infty)$ with $\sum_{i=1}^n p_i^{-1} = 1$.
\end{lemma}
\begin{proof}
    The implication from c) to b) follows from the fact that every continuous negative definite function is quadratically bounded, i.e.\ $|\psi(x)| \leq C(1 + x^2)$ for some $C > 0$, see \cite[Lem.~3.6.22]{Jaco2001}.

    The other implication follows directly from \eqref{eq:prod_est}. To show \eqref{eq:prod_est}, note that the generalized H\"older inequality for $n$-fold products
    (cf.~\cite[p.~133, Pr.~13.5]{schilling-mims})
    gives
    $$
        \Ee\left( \prod_{i=1}^n \psi_i(X_{k_i,i}-X'_{l_i,i})\right)
        \leq \prod_{i=1}^n \left(\Ee \psi_i^{p_i}(X_{k_i,i}-X'_{l_i,i})\right)^{1/p_i}.
    $$
    Using an inequality for continuous negative definite functions (cf.~\cite[Eq.~(2.5)]{part1}, see also \cite[Lem.~3.6.21]{Jaco2001})
    and the Minkowski inequality for the $L^{p_i}$-norm yields the bound
\begin{align*}
        \left(\Ee \psi_i^{p_i}(X_{k_i,i}-X'_{l_i,i})\right)^{1/p_i}
        &\leq 2 \left(\Ee \left[\psi_i(X_{k_i,i})+\psi_i(X'_{l_i,i})\right]^{p_i}\right)^{1/p_i}\\
        &\leq 4 \left(\Ee \psi_i^{p_i}(X_{i})\right)^{1/p_i}. \qedhere
\end{align*}
\end{proof}

\subsection{Proofs and auxiliary results for Section~\ref{sec:theory}}

\begin{proof}[Proof of Proposition~\ref{prop:representation}] Using \eqref{eq:X0X1}, we can rewrite $M_\rho$ in the following way:
\begin{equation}
\begin{aligned}\label{eq:Mrep}
    M_\rho^2 %(X_1,\dots,X_n)
    &= \int \left| \Ee\left[ \prod_{i=1}^n \left(\ee^{\ii \scalp{X_i}{t_i}}- f_{X_i}(t_i)\right)\right] \right|^2 \rho(\ddt)\\
    &= \int \left| \Ee\left[ \prod_{i=1}^n \left(\ee^{\ii \scalp{X_{1,i}}{t_i}}- \ee^{\ii \scalp{X_{0,i}}{t_i}}\right)\right]\right|^2 \rho(\ddt)\\
    &= \int \Ee\left[\prod_{i=1}^n \left(\ee^{\ii \scalp{X_{1,i}}{t_i}} - \ee^{\ii \scalp{X_{0,i}}{t_i}}\right)
    \left(\ee^{-\ii \scalp{X'_{1,i}}{t_i}}- \ee^{-\ii \scalp{X_{0,i}'}{t_i}}\right)\right]\rho(\ddt)\\
    &= \int \Ee\left[\sum_{k,l \in \{0,1\}^{ n}} (-1)^{\sum_{j=1}^n (k_j+l_j)}  \prod_{i=1}^n \ee^{\ii \scalp{(X_{k_i,i}-X'_{l_i,i})}{t_i}}\right] \rho(\ddt)
\end{aligned}
\end{equation}
and the ultimate line already gives \eqref{eq:Msum_1}. By \eqref{eq:Msymmetric},
$$
    M_\rho^2(X_1, X_2, \dots, X_n)
    = \frac{1}{2} \left(M_\rho^2(X_1, X_2, \dots, X_n) + M_\rho^2(-X_1, X_2, \dots, X_n)\right).
$$
Applying this to \eqref{eq:Mrep} shows that the imaginary part of the complex exponential cancels for $i=1$. Repeated applications to $i=2,\dots, n$ removes the other imaginary terms, and we obtain
\begin{equation}\label{eq:M_prod_interm}
    M_\rho^2 %(X_1,\dots,X_n)
    = \int \Ee\left[\sum_{k,l \in \{0,1\}^{ n}} \sgn(k,l) \prod_{i=1}^n \cos(\scalp{(X_{k_i,i}-X'_{l_i,i})}{t_i})\right] \rho(\ddt).
\end{equation}
It remains to show that \eqref{eq:M_prod_interm} is equal to  \eqref{eq:Msum_2}. For this, we note that the product appearing in \eqref{eq:Msum_2} is of the form
$$
    \prod_{i=1}^n \left[\cos(\scalp{(X_{k_i,i}-X'_{l_i,i})}{t_i}) - 1\right]
    = \prod_{i=1}^n \cos(\scalp{(X_{k_i,i}-X'_{l_i,i})}{t_i}) + \prod_{i=1}^n c(k_i,l_i)
$$
where $c(k_i,l_i)$ is either $\cos(\scalp{(X_{k_i,i}-X'_{l_i,i})}{t_i})$ or $-1$ and at least one factor in the second product is $-1$; if, say, $c(k_m,l_m)=-1$ for some $m\in\{1,\dots, n\}$, we get with $k'=(k_1,\dots,k_{m-1},k_{m+1},\dots,k_n)$, $l'=(l_1,\dots,l_{m-1},l_{m+1},\dots,l_n)$,
\begin{align*}
    &\sum_{k,l\in\{0,1\}^n} \sgn(k,l)\prod_{i=1}^nc(k_i,l_i)\\
    &\qquad=  -\sum_{\mathclap{k_m,l_m\in\{0,1\}}} \quad(-1)^{k_m+l_m} \quad\sum_{\mathclap{k',l'\in\{0,1\}^{n-1}}} \quad\sgn(k',l')\prod_{i\neq m} c(k_i,l_i).
\end{align*}
This expression is $0$ since the inner sum does not depend on $k_m,l_m$ and appears exactly four times, twice with positive and twice with negative sign.
This shows that \eqref{eq:Msum_2} is equal to \eqref{eq:M_prod_interm}.

Finally, by Lemma~\ref{lem:moments}, all moment conditions in Definition~\ref{def:moment} imply the mixed moment condition \ref{def:moment}.a), $\Ee\left( \prod_{i=1}^n \psi_i(X_{k_i,i}-X'_{l_i,i})\right)< \infty$ for all $k,l\in\{0,1\}^n$. Under this condition, Fubini's theorem together with the tower property for conditional expectations and the independence properties \eqref{eq:X0X1} of $\bm{X}_0, \bm{X'}_0$ yield
\begin{equation}\begin{aligned}
    M_\rho^2 %(X_1,\dots,X_n)
    &= \Ee\left(\sum_{k,l \in \{0,1\}^{ n}} (-1)^{\sum_{j=1}^n (k_j+l_j)} \prod_{i=1}^n (- \psi_i(X_{k_i,i}-X'_{l_i,i})) \right)\\
&=\label{eq:Mpsi}
 \Ee\left(\prod_{i=1}^n \Psi_{i,0,1}\right)= \Ee\left(\Ee\left(\prod_{i=1}^n \Psi_{i,0,1} \;\middle|\; \bm{X}_1, \bm{X}'_1\right)\right)= \Ee\left(\prod_{i=1}^n \doverline{\Psi}_{i} \right)
\end{aligned}\end{equation}
where
\begin{align*}
\Psi_{i,0,1} := &-\psi_i(X_{1,i}-X'_{1,i})+\psi_i(X_{1,i}-X'_{0,i})\\
&\quad\mbox{}+\psi_i(X_{0,i}-X'_{1,i})-\psi_i(X_{0,i}-X'_{0,i}),\\
\doverline{\Psi}_{i} := &-\psi_i(X_{i}-X'_{i})+\Ee(\psi_i(X_{i}-X'_{i})\mid X_i)\\
&\quad\mbox{}+\Ee(\psi_i(X_{i}-X'_{i})\mid X'_i)-\Ee\psi_i(X_{i}-X'_{i}).\qedhere
\end{align*}
\end{proof}

\subsection{Proofs and auxiliary results for Section~\ref{sec:stats}}

\begin{lemma}\label{lem:ZN}
Let $\bm{X}^{(l)} := (X_1^{(l)},\dots,X_n^{(l)})$ be independent and identical distributed copies of $\bm{X} = (X_1,\dots,X_n)$ and set
\begin{equation}\label{eq:ZN}
    Z_N(t)
    := \frac{1}{N} \sum_{l=1}^N \prod_{i=1}^n \left(\ee^{\ii \scalp{X_i^{(l)}}{t_i}} - \frac{1}{N} \sum_{k=1}^N \ee^{\ii \scalp{X_i^{(k)}}{t_i}}\right).
\end{equation}
Then
\begin{equation} \label{eq:GCNintrep}
    \MN_\rho(\bm{X}^{(1)},\dots,\bm{X}^{(N)})
    = \|Z_N(\sbullet)\|_{L^2(\rho)}.
\end{equation}
If $X_1,\dots,X_n$ are independent, then
\begin{align}
    \Ee Z_N(t) &= 0,\\
    \Ee\big(Z_N(t) \overline{Z_N(t')}\big) &= \frac{1}{N} \cdot C_N\cdot \prod_{i=1}^n \left[ f_{X_i}(t_i-t_i') -  f_{X_i}(t_i) \overline{f_{X_i}(t_i')}  \right],\\
\label{eq:squaredMN}
    \Ee\left(\left|\sqrt{N} Z_N(t) \right|^2\right) &= C_N \cdot \prod_{i=1}^n \left(1 - |f_{X_i}(t_i)|^2 \right),
\end{align}
with constant $C_N:= \frac{(N-1)^n + (-1)^n (N-1)}{N^{n}}.$
\end{lemma}
\begin{proof}
The equality \eqref{eq:GCNintrep} follows by inserting the empirical characteristic function into the representation \eqref{eq:M_rho_L2} of distance multivariance.

Assume that the random variables $X_1, \dots, X_n$ are independent. We obtain
$$
    \Ee\left(\ee^{\ii \scalp{X_i^{(l)}}{t_i}}-\frac{1}{N} \sum_{k=1}^N \ee^{\ii \scalp{X_i^{(k)}}{t_i}}\right)=0,\quad i=1,\dots, n,\; l = 1,\dots,N,
$$
hence, $\Ee Z_N(t) = 0$. Next, consider
\begin{equation} \label{ZN-exp} \begin{aligned}
    \Ee(Z_N(t) \overline{Z_N(t')})
    &= \frac{1}{N^2} \sum_{l,l' =1}^N \Ee\left[ \prod_{i=1}^n \left(\ee^{\ii \scalp{X_i^{(l)}}{t_i}}-\frac{1}{N} \sum_{k=1}^N \ee^{\ii \scalp{X_i^{(k)}}{t_i}}\right)\right.\times\\
    &\qquad\qquad\mbox{}\times\left.\prod_{i'=1}^n \left(\ee^{-\ii \scalp{X_{i'}^{(l')}}{t'_{i'}}}-\frac{1}{N} \sum_{k'=1}^N \ee^{-\ii \scalp{X_{i'}^{(k')}}{t'_{i'}}}\right)  \right].
\end{aligned}\end{equation}
The independence of $X_i$, $X_{j}$ for $i\neq j$ implies
\begin{align*}
    &\Ee\left[ \prod_{i=1}^n \left(\ee^{\ii \scalp{X_i^{(l)}}{t_i}} - \frac{1}{N} \sum_{k=1}^N \ee^{\ii \scalp{X_i^{(k)}}{t_i}}\right)
    \cdot \prod_{i'=1}^n \left(\ee^{-\ii \scalp{X_{i'}^{(l')}}{t'_{i'}}} - \frac{1}{N} \sum_{k'=1}^N \ee^{-\ii \scalp{X_{i'}^{(k')}}{t'_{i'}}}\right)\right]\\
    &\qquad= \prod_{i=1}^n  \Ee\left[ \left(\ee^{\ii \scalp{X_i^{(l)}}{t_i}} - \frac{1}{N} \sum_{k=1}^N \ee^{\ii \scalp{X_i^{(k)}}{t_i}}\right)
        \cdot \left(\ee^{-\ii \scalp{X_{i}^{(l')}}{t'_{i}}} - \frac{1}{N} \sum_{k'=1}^N \ee^{-\ii \scalp{X_{i}^{(k')}}{t'_{i}}}\right)  \right]
\end{align*}
and each factor simplifies to
\begin{align*}
\Ee&\left[ \left(\ee^{\ii \scalp{X_i^{(l)}}{t_i}} - \frac{1}{N} \sum_{k=1}^N \ee^{\ii \scalp{X_i^{(k)}}{t_i}}\right)
        \cdot \left(\ee^{-\ii \scalp{X_{i}^{(l')}}{t'_{i}}} - \frac{1}{N} \sum_{k'=1}^N \ee^{-\ii \scalp{X_{i}^{(k')}}{t'_{i}}}\right)  \right]\\
&= \Ee \ee^{\ii  \scalp{X_i^{(l)}}{t_i}-\ii  \scalp{X_i^{(l')}}{t_i'}} - 2 \frac{N-1}{N} f_{X_i}(t_i) \overline{f_{X_i}(t_i')} - \frac{2}{N} f_{X_i}(t_i-t_i') \\
&\qquad\mbox{}+ \frac{N^2-N}{N^2}  f_{X_i}(t_i) \overline{f_{X_i}(t_i')} + \frac{N}{N^2} f_{X_i}(t_i-t_i')\\
&= \Ee \ee^{\ii \scalp{X_i^{(l)}}{t_i}-\ii  \scalp{X_i^{(l')}}{t_i'}} - \frac{N-1}{N} f_{X_i}(t_i) \overline{f_{X_i}(t_i')} - \frac{1}{N} f_{X_i}(t_i-t_i').
\end{align*}
Thus, splitting the sum in \eqref{ZN-exp} into $l=l'$ and $l\neq l'$ yields
\begin{align*}
    &\Ee\big(Z_N(t) \overline{Z_N(t')}\big)\\
    &= \frac{1}{N^2} \sum_{l,l' =1}^N \prod_{i=1}^n \left[\Ee \ee^{\ii \scalp{X_i^{(l)}}{t_i} - \ii  \scalp{X_i^{(l')}}{t_i'}}
        - \frac{N-1}{N} f_{X_i}(t_i) \overline{f_{X_i}(t_i')} - \frac{1}{N} f_{X_i}(t_i-t_i')\right]\\
    &= \frac{N}{N^2} \prod_{i=1}^n \left[-\frac{N-1}{N}  f_{X_i}(t_i) \overline{f_{X_i}(t_i')} - \left(\frac{1}{N}-1\right) f_{X_i}(t_i-t_i') \right] \\
    &\qquad\mbox{} + \frac{N^2-N}{N^2} \prod_{i=1}^n \left[\left(-\frac{N-1}{N} +1\right) f_{X_i}(t_i) \overline{f_{X_i}(t_i')} - \frac{1}{N} f_{X_i}(t_i-t_i') \right]\\
    &= \left(\frac{1}{N} \left(\frac{N-1}{N}\right)^n + \frac{N-1}{N} \left(\frac{-1}{N}\right)^n \right) \prod_{i=1}^n \left[f_{X_i}(t_i-t_i') -  f_{X_i}(t_i) \overline{f_{X_i}(t_i')} \right]\\
    &= \frac{(N-1)^n + (-1)^n (N-1)}{N^{n+1}} \prod_{i=1}^n \left[ f_{X_i}(t_i-t_i') -  f_{X_i}(t_i) \overline{f_{X_i}(t_i')}  \right].
\end{align*}
For $t'=t$ this reduces to
\begin{equation*}
    \Ee\left(\left|\sqrt{N} Z_N(t) \right|^2\right) = N\cdot \frac{(N-1)^n + (-1)^n (N-1)}{N^{n+1}}\prod_{i=1}^n \left(1 - |f_{X_i}(t_i)|^2 \right).\qedhere
\end{equation*}
\end{proof}

\begin{proof}[Proof of Theorem~\ref{thm:Mdistconv}]
We start with part b), which is a simple consequence of the strong consistency of $\MN_\rho$. Indeed, by Theorem~\ref{thm:sconsistent} we have $\MN_\rho \to M_\rho$ a.s., and from Theorem~\ref{thm:independence} we know that $M_\rho > 0$ under the conditions of b), such that \eqref{eq:Mdistconv_b} follows.

For part a), let $Z_N(t)$ be defined as in \eqref{eq:ZN}. Then $\MN_\rho(\bm{X}^{(1)},\dots,\bm{X}^{(N)}) =  \|Z_N(.)\|_{L^2(\rho)}$ by Lemma~\ref{lem:ZN}.

If $\sqrt{N} Z_N$ converges in distribution to a Gaussian process then, by Lemma \ref{lem:ZN}, this process is centred and has the covariance structure \eqref{eq:G_cov}, i.e.\ it is distributed as $\mathds{G}$.
In order to show convergence, we introduce the following notation. Denote by $F_{\bm{X}}$ the distribution function of $\bm{X}$ and by $\FN_{\bm{X}}$ the empirical distribution function of the iid sequence $(\bm{X}^{(1)}, \dots, \bm{X}^{(N)})$. For a subset $S \subset \{1, \dots, n\}$ we write $t_S := (t_i)_{i \in S}$ and denote the corresponding empirical characteristic function by
\begin{equation*}
    \fN_S(t_S) := \frac{1}{N} \sum_{j=1}^N \exp\left(\ii \sum_{i \in S} \scalp{X_i^{(j)}}{t_i} \right) = \int\ee^{\ii \scalp{x_S}{t_S}}\,\dd(\FN_{\bm{X}}(x)).
\end{equation*}
If $S=\{i\}$ is a singleton, we write $\fN_i :=  \fN_{\{i\}}$.
By \cite[Thm~3.1, p.~208]{Csoe1981a} the $\log$-moment condition is sufficient for the convergence
\begin{equation}\label{eq:convBB}
    {\scriptstyle\sqrt{N}} \left(\fN(t) - f(t)\right)
    = \int\ee^{\ii \scalp{x}{t}}\,\dd\!\left( {\scriptstyle\sqrt{N}}(\FN_{\bm{X}}(x) -F_{\bm{X}}(x))\right)
    \xrightarrow[N\to \infty]{d} \int\ee^{\ii \scalp{x}{t}}\,\ddB(x),
\end{equation}
where $B$ is a Brownian bridge indexed by $\real^d$ (cf.~\cite[Eq.~(3.2)]{Csoe1981a}) and the distributional convergence is uniform (in $t$) on compact subsets of $\real^d$.
Next, we rewrite $Z_N$ from \eqref{eq:ZN} as
\begin{equation}\label{eq:ZN2}
    Z_N(t) = \sum_{S\subset \{1,\dots,n\}} (-1)^{n-|S|} \left(\fN_S(t_S) \cdot \prod_{\smash{j} \in S^c} \fN_j(t_j)\right).
\end{equation}
In addition, we have the simple identity, cf.~\eqref{eq:ab_identity},
\begin{equation}\label{eq:prod}
\prod_{j=1}^n \left(f_j(t_j) - \fN_j(t_j) \right) = \sum_{S\subset \{1,\dots,n\}} (-1)^{n-|S|} \left(\prod_{\smash{j} \in S} f_j(t_j) \cdot \prod_{\smash{j} \in S^c} \fN_j(t_j)\right).
\end{equation}
Subtracting \eqref{eq:prod} from \eqref{eq:ZN2} and rearranging the resulting equation yields
\begin{align*}
\sqrt{N} Z_N(t)
&= \sum_{S\subset \{1,\dots,n\}} (-1)^{n-|S|} \sqrt{N} \big(\fN_S(t_S) - f_S(t_S)\big) \cdot \prod_{\smash{j} \in S^c} \fN_j(t_j) \\
&\qquad\mbox{}+ \sqrt{N} \prod_{j=1}^n \big(f_j(t_j)- \fN_j(t_j)\big).
\end{align*}
By \eqref{eq:convBB}, we have that
$$
    \sqrt{N} \left(\fN_S(t_S) - f_S(t_S)\right) \xrightarrow[N\to \infty]{d} \int\ee^{\ii \scalp{x_S}{t_S}}\,\ddB(x).
$$
By the Glivenko--Cantelli theorem  the limit $\lim_{N\to\infty}\fN_j(t_j) = f_j(t_j)$ exists  uniformly in $t_j$ for all $j = 1, \dots, n$, and thus
\begin{align*}
     &{\sqrt{N}}\prod_{j=1}^n \left(\fN_j(t_j)- f_j(t_j)\right)\\
     &
    =  {\sqrt{N}}\left( \fN_1(t_1) - f_1(t_1)\right) \cdot \prod_{j=2}^n \left(\fN_j(t_j)- f_j(t_j)\right)
    \xrightarrow[N\to\infty]{} 0.
\end{align*}
Together with \eqref{eq:ZN2} this yields the convergence
\begin{equation}\label{eq:conv_to_G}
\sqrt{N} Z_N(t)  \xrightarrow[N\to \infty]{d} \sum_{S\subset \{1,\dots,n\}} (-1)^{n-|S|} \int\ee^{\ii \scalp{x_S}{t_S}}\,\ddB(x) \cdot \prod_{\smash{j} \in S^c} f_j(t_j),
\end{equation}
which takes place uniformly on compacts. The right hand side is a complex-valued Gaussian process indexed by $\real^d$; denoting this process by $\mathds{G}$, we have thus shown that for each $T > 0$,
\begin{equation}\label{eq:sum-conv}
    \sqrt{N} Z_N \xrightarrow[N\to\infty]{d} \mathds{G}
    \quad\text{on}\quad
    \Cskript_T:= (C(B^d_T), \|\sbullet\|_{B^d_T}),
\end{equation}
where $B^d_T := B^d_T(0) := \{x\in \real^d\,:\,|x| <  T\}$ and $\|f\|_{B^d_T} := \sup_{x \in B^d_T }|f(x)|$. To obtain \eqref{eq:G_conv}, it remains to show that also the $L^2(\rho)$-norms of both sides of \eqref{eq:sum-conv} converge, and that $T$ can be sent to infinity.
To this end, we apply a truncation argument.

Set
\begin{equation}
\rho_{i,\epsilon}(A) := \rho_{i}(A\cap (B_{1/\epsilon }^{d_i}\setminus B_{\epsilon}^{d_i})) \quad \text{ and } \quad  \rho_i^{\epsilon} := \rho_i-\rho_{i,\epsilon},
\end{equation}
and note that the $\rho_{i,\epsilon}$ are finite measures for  each  $\epsilon >0$, by \eqref{eq:levy_integrability}. In addition, we define $\rho_{\epsilon} = \bigotimes_{i=1}^n\rho_{i,\epsilon} $ as well as $ \rho^{\epsilon} =  \bigotimes_{i=1}^n\rho_i^{\epsilon}$ and introduce, for this proof, the shorthand notation $\norm{\sbullet}_{\rho_\epsilon} = \norm{\sbullet}_{L^2(\rho_\epsilon)}$. Note that $|x_i| \le 1/\epsilon $, $x_i\in\real^{d_i}$, for all $i=1,\dots, n$ implies $|x| \le \sqrt{n}/\epsilon$, $x=(x_1,\dots,x_n)\in\real^d$, and hence we have
\begin{equation}
    \left|\|h\|_{\rho_{\epsilon}} - \|h'\|_{\rho_{\epsilon}} \right|^2
    \leq \left\|h-h'\right\|_{\rho_\epsilon}
    \leq \sup_{|x|\,\leq \sqrt n/\epsilon} \left|h(x)-h'(x)\right|^2 \cdot \prod_{i=1}^n \rho_{i,\epsilon}(\real^{d_i}),
\end{equation}
which shows that $\|\sbullet\|^2_{\rho_{\epsilon}}$ is continuous on $\Cskript_T$ for any $T \geq \sqrt{n}/\epsilon$.  Thus, the continuous mapping theorem implies that for any $\epsilon > 0$
\begin{equation}\label{eq:G_cmt}
    \|\sqrt{N} Z_N\|^2_{\rho_{\epsilon}} \xrightarrow[N\to\infty]{d} \|\mathds{G}\|^2_{\rho_{\epsilon}}.
\end{equation}
By the portmanteau theorem, the convergence \eqref{eq:G_conv} is equivalent to the statement
$$
    \lim_{N \to \infty} \Ee\left(h(N\cdot \MN^2) - h(\norm{\mathds{G}}^2_\rho)\right) = 0
$$
for all bounded Lipschitz continuous functions $h: \real \to \real$. Denoting the Lipschitz constant of $h$ by $L_h$, we see
\begin{align}\notag
    \left|\Ee h(N\cdot \MN^2) - \Ee h(\norm{\mathds{G}}^2_\rho) \right|
    &\leq L_h \Ee \left| N\cdot \MN^2 - \|\sqrt{N} Z_N\|^2_{\rho_\epsilon} \right|\\
    &\label{eq:G_split}\qquad\mbox{}+ \left|\Ee h\left(\|\sqrt{N} Z_N\|^2_{\rho_\epsilon}\right) - \Ee h\left(\norm{\mathds{G}}^2_{\rho_\epsilon}\right)\right|\\
    &\notag\qquad\mbox{}+ L_h \Ee \left|\norm{\mathds{G}}^2_{\rho_\epsilon} -\norm{\mathds{G}}^2_{\rho}\right|.
\end{align}
The middle term tends to zero as $N \to \infty$, by \eqref{eq:G_cmt}. To estimate the other terms, define $\mu^\epsilon$ to be the measure given by
\begin{equation*}
    \left(\rho_1^{\epsilon} \otimes \rho_2 \otimes {\dots} \otimes \rho_n \right)+
     \left(\rho_1 \otimes \rho_2^{\epsilon} \otimes {\dots} \otimes \rho_n \right)  +
    \dots + \left( \rho_1 \otimes {\dots} \otimes \rho_{n-1} \otimes \rho_n^{\epsilon} \right).
\end{equation*}
For the first term on the right hand side of \eqref{eq:G_split} we get the bound
\begin{equation*}
    \Ee \left|N \cdot \MN^2 - \|\sqrt{N} Z_N\|^2_{\rho_{\epsilon}}\right|
    = \Ee \left|\|\sqrt{N} Z^N\|_{\rho}^2 - \|\sqrt{N} Z_N\|^2_{\rho_{\epsilon}}\right|
    \leq \Ee \|\sqrt{N} Z_N\|^2_{\mu^\epsilon}.
\end{equation*}
Using \eqref{eq:squaredMN} we see with $C_N := [(N-1)^n+(-1)^n(N-1)]/N^n \leq 1$
\begin{equation}
\label{eq:Zn-on-bounded}
    \left\|\Ee \left(|\sqrt{N} Z_N|^2\right)\right\|^2_{ \mu^\epsilon}
    = C_N \sum_{k=1}^n \bigg[\|1-|f_{X_k}|^2\|^2_{ \rho_k^{\epsilon}} \prod_{\substack{i=1 \\ i\neq k}}^n \|1-|f_{X_i}|^2\|^2_{\rho_i}\bigg],
\end{equation}
and this expression converges to $0$ as $\epsilon\to 0$. This follows from dominated convergence, since
\begin{equation*}
    \sum_{k=1}^n \bigg[\|1-|f_{X_k}|^2\|^2_{ \rho_k^{\epsilon}} \prod_{\substack{i=1 \\ i\neq k}}^n \|1-|f_{X_i}|^2\|^2_{\rho_i}\bigg]
    \leq n \prod_{i=1}^n \Ee\psi_i(X_i-X_i') < \infty.
\end{equation*}
The last term in \eqref{eq:G_split} can be estimated in a similar way. We have
\begin{equation} \label{eq:gauss-on-bounded}
    \|\mathds{G}\|^2_{\rho} - \|\mathds{G}\|^2_{\rho_{\epsilon}}
    \leq\|\mathds{G}\|^2_{\mu^\epsilon}
    \xrightarrow[\epsilon \to 0]{} 0
    \quad\text{a.s.}
\end{equation}
by dominated convergence, since $\lim_{\epsilon\to 0}\int g_i \, d\rho_i^{\epsilon} = 0$ for integrable $g_i$ and
\begin{equation}\label{eq:EG-2}\begin{aligned}
    \Ee(\|\mathds{G}\|^2_{\mu^\epsilon})
    \leq n \Ee(\|\mathds{G}\|^2_{\rho})
    &= n \prod_{i=1}^n \|1-|f_{X_i}|^2\|^2_{\rho_i}\\
    &= n \prod_{i=1}^n \Ee\psi_i(X_i-X_i') < \infty.
\end{aligned}\end{equation}
Together with \eqref{eq:G_split} this shows the convergence result \eqref{eq:G_conv} and completes the proof. \end{proof}

%\bibliographystyle{alpha}
%\bibliography{./distance-covariance-part2}

\vfill
\noindent B.~B\"{o}ttcher, M.~Keller-Ressel, R.\,L.~Schilling\\
bjoern.boettcher@tu-dresden.de\\
martin.keller-ressel@tu-dresden.de\\
rene.schilling@tu-dresden.de\\ \\
TU Dresden \\ Fakult\"at Mathematik\\ Institut f\"{u}r Mathematische Stochastik\\ 01062 Dresden, Germany

\end{document}